\documentclass[draft]{amsart}

\usepackage{amssymb}
\usepackage{url}
\usepackage{amscd}

\usepackage{amssymb} 
\usepackage{amsthm}
\usepackage{amsmath}

\usepackage{latexsym}
\usepackage{mathrsfs}
\usepackage{enumerate}
\usepackage[dvips]{graphicx}
\usepackage{color}

\definecolor{mgre}{cmyk}{0.92,0.00,0.59,0.25}

\def\cD{{\mathcal{D}}}
\def\cE{{\mathcal{E}}}
\def\cF{{\mathscr{F}}}

\def\cH{{\mathcal{H}}}

\def\cK{{\mathcal{K}}}
\def\cL{{\mathcal{L}}}

\def\sB{{\mathscr{B}}}

\def\sF{{\mathscr{F}}}

\def\1{{\mathbf{1}}}
\newcommand{\Sch}{Schr\"odinger}

\newcommand{\capa}{\mathrm{Cap}}

\newcommand{\loc}{\textrm{loc}}

\def\<{\langle}
\def\>{\rangle}
\newtheorem{thm}{Theorem}[section]

\newtheorem{prop}[thm]{Proposition}
\newtheorem{lem}[thm]{Lemma}
\newtheorem{cor}[thm]{Corollary}

\theoremstyle{definition}
\newtheorem{defi}[thm]{Definition}
\newtheorem{prob}{}[section]
\newtheorem{rem}[thm]{Remark}
\newtheorem{ex}[thm]{Example}
\newtheorem{pro}{Problem}[section]
\newtheorem{ass}[thm]{Assumption}

\newcommand{\bd}{\begin{defi}}
\newcommand{\ed}{\end{defi}}

\newcommand{\bpro}{\begin{pro}}
\newcommand{\epro}{\end{pro}}

\newcommand{\bec}{\begin{cases}}
\newcommand{\eec}{\end{cases}}

\newcommand{\bpr}{\begin{prob}}
\newcommand{\epr}{\end{prob}}

\newcommand{\bt}{\begin{thm}}
\newcommand{\et}{\end{thm}}

\newcommand{\ba}{\begin{ass}}
\newcommand{\ea}{\end{ass}}

\newcommand{\br}{\begin{rem}}
\newcommand{\er}{\end{rem}}

\newcommand{\bpm}{\begin{pmatrix}}
\newcommand{\epm}{\end{pmatrix}}

\newcommand{\be}{\begin{ex}}
\newcommand{\ee}{\end{ex}}

\newcommand{\bp}{\begin{prop}}
\newcommand{\ep}{\end{prop}}

\newcommand{\bl}{\begin{lem}}
\newcommand{\el}{\end{lem}}

\newcommand{\bc}{\begin{cor}}
\newcommand{\ec}{\end{cor}}

\newcommand{\bq}{\begin{que}}
\newcommand{\eq}{\end{que}}

\newcommand{\beqn}{\begin{eqnarray*}}
\newcommand{\eeqn}{\end{eqnarray*}}

\newcommand{\beqnn}{\begin{eqnarray}}
\newcommand{\eeqnn}{\end{eqnarray}}

\newcommand{\bequ}{\begin{equation}}
\newcommand{\eequ}{\end{equation}}

\newcommand{\benu}{\begin{enumerate}}
\newcommand{\eenu}{\end{enumerate}}

\newcommand{\barr}{\begin{array}{rcl}}
\newcommand{\ear}{\end{array}}

\newcommand{\la}{\label}

\newcommand{\ds}{\displaystyle}
\newcommand{\eps}{\epsilon}

\newcommand{\Om}{\Omega}

\usepackage{amsmath}	

 \makeatletter
    
    \@addtoreset{equation}{section}
  \makeatother
\begin{document}

\title[Criticality and Recurrence]{Criticality of Schr\"{o}dinger Forms 
and Recurrence of Dirichlet Forms
}

\author{Masayoshi Takeda, Toshihiro Uemura} 

\address{Department of Mathematics,
		Kansai University, Yamatecho, Suita, 564-8680, Japan}

\email{mtakeda@kansai-u.ac.jp} 

\address{Department of Mathematics,
		Kansai University, Yamatecho, Suita, 564-8680, Japan}

\email{t-uemura@kansai-u.ac.jp} 

\thanks{The first author was supported in part by Grant-in-Aid for Scientific
	Research (No.18H01121(B)), Japan Society for the Promotion of Science. 
	The second author was supported in part by Grant-in-Aid for Scientific 
		Research (No.19K03552(C)), Japan Society for the Promotion of Science.}


\begin{abstract}
Introducing the notion of extended Schr\"odinger spaces, we define the criticality and subcriticality 
of Schr\"odinger forms 
  in the same manner as the recurrence and transience of Dirichlet forms, and give a sufficient condition for the 
subcriticality of Schr\"odinger forms in terms the bottom of spectrum. 
We define a subclass of Hardy potentials and prove that Schr\"odinger forms with potentials in this subclass
 are always critical, which leads us to optimal Hardy type inequality.  
We show that this definition of criticality and subcriticality 
 is equivalent to that there exists an excessive 
function with respect to Schr\"odinger semigroup and its generating Dirichlet form through $h$-transform 
 is recurrent and transient respectively.
As an application, we can show the recurrence and transience of a family of Dirichlet forms by showing 
the criticality and subcriticaly of Schr\"odinger forms and show the other way around through 
$h$-transform, We give a such example with fractional Schrödinger operators with Hardy potential.
\end{abstract}

\maketitle

\newcommand\sfrac[2]{{#1/#2}}

\newcommand\cont{\operatorname{cont}}
\newcommand\diff{\operatorname{diff}}

\section{Introduction}
Let $E$ be a locally compact separable metric space and $m$ a positive
Radon measure on $E$ with full topological support. 
Let $X= (P_x, X_t, \zeta)$ be an $m$-symmetric Hunt process. We assume that 
$X$ is irreducible, strong Feller and transient, in addition, that 
 the Dirichlet form $(\cE, \cD(\cE))$ generated by $X$ is {regular} 
on $L^2(E;m)$.  
For a Green-tight Kato measure $\mu$ ($\mu\in\cK_\infty$ in notation), 
define a Schr\"{o}dinger form $(\cE^\mu , \cD(\cE^\mu ))$ by  
\begin{equation}\la{s-1}
\cE^\mu(u,u)=\cE(u,u)-\int_Eu^2d\mu,\ \ u\in\cD(\cE^\mu)(=\cD(\cE)).
\end{equation}
One of authors define in \cite{T3} the criticality or subcriticality for  $\cE^\mu$ 
through $h$-transform; let $A^{\mu}_t$ be the positive continuous additive functional 
with Revuz measure $\mu$ and $\{p^\mu_t\}_{t\geq 0}$ the Feynman-Kac semigroup defined by
\bequ\la{f-k}
p^\mu_tf(x)=E_x\left(e^{A^\mu_t}f(X_t)\right).
\eequ
We introduce the space of $p^\mu_t$-excessive functions by
	\begin{equation}\la{mu-exc}
		\cH^+(\mu) = \{ h \,\mid\, h\ \text{is\ quasi\text{-}continuous},\ h>0
		\ {\rm q.e.},\   p_t^{\mu} h \le h\  {\rm q.e.}\}.
	\end{equation}
Here the term ``q.e.'' means ``except on a set of zero capacity''.
Suppose that $\cH^+(\mu)$ is not empty. Then for $h \in \cH^+(\mu)$ the symmetric form $(\cE^{\mu,h},
\cD(\cE^{\mu,h}))$ on $L^2(E;h^2m)$ is defined by $h$-transform of $(\cE^\mu, \cD(\cE^\mu))$:
\begin{equation}\la{h-form}
\left\{
\begin{split}
& \cE^{\mu,h} (u,u)  = \cE^{\mu} (hu, hu)  \\
& \cD(\cE^{\mu,h}) = \{ u \in L^2(E;h^2m)\mid hu\in \cD(\cE^\mu)\}.
\end{split} 
\right.
\end{equation}
Li  show in \cite{Li} that $(\cE^{\mu,h},\cD(\cE^{\mu,h}))$ turns out to be a {\it quasi-regular 
Dirichlet form} on $L^2(E;h^2m)$. 	
Consequently, if $\cH^+(\mu)$ is not empty, then $(\cE^\mu , \cD(\cE^\mu ))$ is positive semi-definite,
$\cE^\mu(u,u)\geq 0$ for all $u\in \cD(\cE^\mu )$.
The $L^2(E;h^2m)$-Markov semigroup $\{T_t^{\mu,h}\}_{t\geq 0}$ associated 
with $(\cE^{\mu,h},\cD(\cE^{\mu,h}))$ is expressed by 
$$
T_t^{\mu,h}f(x)=\frac{1}{h(x)}p_t^\mu(h f)(x)\  \ m\text{-a.e.}
$$
In \cite{Li}, \cite{T3}, we define the {subcriticality} (resp. {criticality}) for $(\cE^{\mu}, \cD(\cE^{\mu}))$ 
 as follows: $(\cE^{\mu}, \cD(\cE^{\mu}))$ is said to
	be {critical} (resp. {subcritical}) if  $\cH^+(\mu)$ is not empty and 
	$(\cE^{\mu,h}, \cD(\cE^{\mu,h}))$ is recurrent
(resp.  transient) for some $h \in \cH^+(\mu)$. Besides, $(\cE^{\mu}, \cD(\cE^{\mu}))$ is  
	said to be {supercritical} if $\cH^{+}(\mu)$ is empty. We show that these definitions are well-defined and obtain the following analytic criteria for the {subcriticality} and  {criticality} (\cite[Theorem 5.19]{T3}): Define $\lambda (\mu)$ by
\begin{equation}\la{j-ground1}
\lambda (\mu) := \inf \left\{ \cE(u)\, \Big|\, u \in
		       \cD(\cE)\cap C_0(E),\ \int_E u^2 d\mu = 1\right\}.
\end{equation}
(We simply write $a(u)$ for $a(u,u)$ for a symmetric form $a(u,v)$). Here, $\lambda(\mu)$ is regarded as 
the principal eigenvalue of the time changed process by $A^\mu_t$. Then
$(\cE^{\mu},\cD(\cE^{\mu}))$ is subcritical, critical and supercritical,
 if and only if $\lambda (\mu)>1$, $\lambda (\mu)=1$ and $\lambda (\mu)<1$ respectively. 
 For establishing this analytic criteria, the assumption $\mu\in\cK_\infty$ is crucial. 
 Indeed,  since for $\mu\in\cK_\infty$ the extended Dirichlet space 
 of $(\cE,\cD(\cE))$ is compactly embedded in $L^2(E;\mu)$ (\cite[Theorem 4.8]{T-C}) and thus  
 the minimizer of (\ref{j-ground1}) exists. By the argument similar to  one in the proof of \cite[Theorem 3.1]{T5}
  we see that the minimizer
 belongs to $\cH^+(\mu)$ if and only if $\lambda(\mu)\ge 1$, i.e.
\bequ\la{m-all}
\lambda(\mu)\ge 1\ \Longleftrightarrow\ \cH^+(\mu)\not=\emptyset.
\eequ
Let $\gamma(\mu)$ be the bottom of the Schr\"{o}dinger form   $(\cE^\mu , \cD(\cE^\mu ))$: 
\bequ\la{bot}
\gamma(\mu)=\inf\left\{\cE^\mu(u)\, \Big|\, u\in\cD(\cE)\cap C_0(E),\ \int_E u^2 dm = 1\right\}.
\eequ
Since $\lambda(\mu)\ge 1$ is equivalent to $\gamma(\mu)\ge 0$ (\cite[Lemma 2.2]{TT}),  the following Allegretto-Piepenbrink-type theorem is obtained: For $\mu\in\cK_{\infty}$
\bequ\la{a-p}
\gamma(\mu)\ge 0\ \Longleftrightarrow\ \cH^+(\mu)\not=\emptyset.
\eequ
This is a reason why for $\mu\in\cK_{\infty}$ the criticality or subcriticality of $(\cE^\mu , \cD(\cE^\mu ))$ are completely classified in terms of   
$\lambda(\mu)$. 

\smallskip
In this paper, we treat more general class of measures including Hardy class of potentials, 
and define  the {criticality}
 (resp. {subcriticality}) of associated Schr\"odinger forms  in another way similar to 
the recurrence (resp. transience) of Dirichlet forms, without using the $h$-transform. More precisely, 
Let $(\cE,\cD(\cE))$ be a regular Dirichlet form. We assume that it is irreducible and transient throughout this paper.
For a smooth Radon measure $\mu$, define the form $(\cE^\mu,\cD(\cE)\cap C_0(E))$ as in (\ref{s-1}). 
We assume that 
 $(\cE^\mu,\cD(\cE)\cap C_0(E))$ 
is positive semi-definite, i.e., $\gamma(\mu)$ in (\ref{bot}) is non-negative,  
and closable in $L^2(E;m)$. We denote its closure by 
$(\cE^\mu , \cD(\cE^\mu ))$ (For a sufficient condition for the closability, see Lemma \ref{e-1}, 
Theorem \ref{cl} below). 

The criticality and subcriticality of Schr\"odinger forms are extended notions of the recurrence and transience 
of Dirichlet forms. For characterizing the recurrence and transience 
of Dirichlet forms the extended Dirichlet spaces play a crucial role. 
In Schmuland \cite{Sch} and  \cite{Sch1}, he extends the notion of extended Dirichlet spaces 
 to positive semi-definite symmetric closed forms with {\it Fatou property} and 
 show that positivity preserving forms have
  the Fatou property. We denote by ${\cD_e}(\cE^{\mu})$
 the extended space of $(\cE^{\mu}, \cD(\cE^{\mu}))$ and call it {\it extended Schr\"odinger space}
  (The definition of extended Schr\"odinger spaces is given in Section 2). 
On account of these facts, we easily imagine that the notion of extended Schr\"odinger spaces is available
 for characterizing the criticality and subcriticality of Schr\"odinger forms.  Indeed, we define a 
 form $(\cE^{\mu}, \cD(\cE^{\mu}))$ is {\it subcritical} if there exists a bounded function $g\in L^1(E;m)$
  strictly positive $m$-a.e. such that 
\bequ\la{n-sub}
	\int_E|u| g dm\le \sqrt{\cE^\mu(u)},\ \ u\in\cD_e(\cE^{\mu})
\eequ
and is {\it critical} if there exists a function $\phi$ in $\cD_e(\cE^{\mu})$ strictly positive $m$-a.e. 
	 such that $\cE^\mu(\phi)=0$. In other words, the function $\phi$ is the ground state of the  Schr\"odinger 
	 operator $\cH^\mu$ associated with $(\cE^{\mu}, \cD(\cE^{\mu}))$, $(-\cH^\mu u,v)_m=\cE^\mu(u,v)$.
We would like to emphasize that the space ${\cD_e}(\cE^{\mu})$ contains the extended Dirichlet 
	 space ${\cD_e}(\cE)$ of the Dirichlet space $(\cE,\cD(\cE))$ and  $\phi$ does not belong to ${\cD_e}(\cE)$
	 generally. 

Let $\{T^\mu_t\}$	 be the $L^2(E;m)$-semigroup associated with $(\cE^{\mu}, \cD(\cE^{\mu}))$. 
For a general smooth Radon measure $\mu$, we do not know whether 
the semigroup $T^\mu_t$ is probabilistically expressed as 
\bequ\la{p-e}
T^\mu_tf(x)=p^\mu_tf(x)\ \ m\text{-a.e. },
\eequ
where the semigroup $p^\mu_t$ is defined in (\ref{f-k}), and thus define the space of $T^\mu_t$-excessive 
functions by
	\begin{equation}\la{mu-exc}
		\cH^+(\mu) = \{ h \,\mid\, 0<h<\infty\ m\text{-a.e.},\   T_t^{\mu} h \le h\ m\text{-a.e.}\}.
	\end{equation}
Here note that $T^\mu_t$ can be extended to an operator on the space of non-negative functions.

 For $h \in \cH^+(\mu)$ let $(\cE^{\mu,h},
\cD(\cE^{\mu,h}))$ be the Dirichlet form on $L^2(E;h^2m)$ defined by $h$-transform 
of $(\cE^\mu, \cD(\cE^\mu))$ as (\ref{h-form}).
We then show that if $(\cE^{\mu}, \cD(\cE^{\mu}))$ is subcritical, 
then the $0$-resolvent $G^\mu g(x)=\int_0^\infty T^\mu_tg(x)dt$ of 
the function $g$ in (\ref{n-sub}) belongs to 
$\cD_e(\cE^\mu)\cap \cH^+(\mu)$ and  
 $(\cE^{\mu,G^\mu g}, \cD(\cE^{\mu,G^\mu g}))$ has the killing
part, in particular, transient (Lemma \ref{77}, Lemma \ref{ma} below). 
Consequently, for any $h\in\cH^+(\mu)$ $(\cE^{\mu,h}, \cD(\cE^{\mu,h}))$ is 
transient (Lemma \ref{well1}). 

Suppose that $\mu$ satisfies $\lambda(\mu)>1$, where $\lambda(\mu)$ is the bottom of spectrum in 
(\ref{j-ground1}). \footnote{For a general smooth Radon measure $\mu$ 
we do not know whether $\lambda(\mu)$ is the principal eigenvalue or not. }
We then show that $G^\mu\varphi$ is in $\cD_e(\cE^\mu)\cap \cH^+(\mu)$ 
for any non-trivial, non-negative function $\varphi\in C_0(E)$, which 
 implies $(\cE^{\mu},\cD(\cE^{\mu}))$ is subcritical. Note that $\mu$ is not supposed 
to be Green-tight. Conversely, we can show if there exists a $h\in\cH^+(\mu)$ 
such that $(\cE^{\mu,h}, \cD(\cE^{\mu,h}))$ is 
transient, then $(\cE^\mu,\cD(\cE^\mu))$ is subcritical.

We show that if $(\cE^{\mu}, \cD(\cE^{\mu}))$ is critical, 
then the ground state $\phi$ in the definition of the criticality satisfies $T^\mu_t\phi=\phi$, in particular, 
$\phi\in\cH^+(\mu)$,  and   
$(\cE^{\mu,\phi},\cD(\cE^{\mu,\phi}))$ is recurrent. Moreover, every function $h\in\cH^+(\mu)$ can be written
as $h=c\phi,\ c>0$. Conversely, if there exists a $h\in\cH^+(\mu)$ such that $(\cE^{\mu,h}, \cD(\cE^{\mu,h}))$ is 
recurrent, then $(\cE^\mu,\cD(\cE^\mu))$ is critical. In Section 6, 
we treat a family of Dirichlet forms defined 
in (\ref{del}) $(0\le \delta\le d-\alpha)$ below and show that it is recurrent if and only if $\delta=(d-\alpha)/2$.

Here we would like to make a comment on problems caused by treating a general smooth Radon measure instead 
of a Green-tight Kato measure: When $\lambda(\mu)=1$, we do not know how to construct 
a function in $\cH^+(\mu)$ and whether a Schr\"odinger form is subcritical or critical. In particular, we 
do not know whether the Allegretto-Piepenbrink-type theorem (\ref{a-p}) holds or not
 (For recent results on Allegretto-Piepenbrink-type theorem, see \cite{JMV}, \cite{KP}, \cite{LS}). 
 Consequently, we cannot classify positive semi-definite Schr\"odinger forms to subcritical and critical ones
 in terms of $\lambda(\mu)$. 

 From Section 4, we treat the symmetric Hunt process $X$ generated by the regular Dirichlet space 
 $(\cE,\cD(\cE))$.  By the assumption on $(\cE,\cD(\cE))$, it is irreducible and transient. We assume, 
 in addition, that it 
 possesses the strong Feller property.
As remarked above,
 we do not know whether $T^\mu_t$ is probabilistically expressed as (\ref{p-e}).
 Suppose that $\mu$ is a smooth measure in local Kato class ($\mu\in\cK_{\loc}$ in notation),
  that is, for any compact set $F\subset E$ the restriction
 of $\mu$ to $F$ belongs to the Kato class $\cK$ 
 (Definition \ref{def-Kato}).\footnote{By the regularity of $(\cE,\cD(\cE))$, 
 a measure in $\cK_{\loc}$ is Radon.}
 Using facts in Albeverio and Ma \cite{AM} we can show that $(\cE^\mu,\cD(\cE)\cap C_0(E))$ 
 is closable and its semigroup $T^\mu_t$ has a probabilistic representation (\ref{p-e}). 
 As a result, we prove in Theorem \ref{Id} and Theorem \ref{cl} that
  the $(\cE^{\mu}, \cD(\cE^{\mu}))$ is regarded as the closed form generated by $p^\mu_tf(x)$.
These facts can be extended to a smooth measure $\mu$ such that 
there exists a compact set $K$ with Cap($K)=0$ 
and $\mu\in \cK^D_{\loc}$, $D=E\setminus K$, that is, for any compact set $F\subset D$ the restriction
 of $\mu$ to $F$ belongs to the Kato class $\cK$. Here Cap is the capacity associated with the regular Dirichlet 
 form $(\cE,\cD(\cE))$. Furthermore, we see that the closure of $(\cE^\mu,\cD(\cE)\cap C_0(D))$ is
 identified with that of $(\cE^\mu,\cD(\cE)\cap C_0(E))$.

We denote by $R(x,y)$ the 0-resolvent kernel of $X$ and define the potential of a measure $\mu$ by
$$
R\mu(x)=\int_ER(x,y)d\mu(y).
$$
 We introduce a subclass $\cK_{H}$ of $\cK_{\loc}$ as follows: A measure $\mu\in\cK_{\loc}$
 belongs to $\mu\in\cK_{H}$ if $\mu$ satisfies that $R\mu$ is in $\cD_{\loc}(\cE)\cap \sB_{b,\loc}(E)$ and  
 there exists an
 increasing sequence $\{K_n\}$ of compact sets such that $K_n\uparrow E$ and 
 \bequ\la{kh}
 \sup_n\iint_{K_n\times K_n^c}R(x,y)d\mu(x)d\mu(y)<\infty.
\eequ
Here $\cD_{\loc}(\cE)\cap \sB_{b,\loc}(E)$ is the set of functions locally in $\cD(\cE)\cap\sB_b(E)$. 
A measure $\mu\in\cK_{\loc}$ of finite energy integral belongs to $\cK_H$.  
 For $\mu\in\cK_H$ define a measure $\nu$ by $\nu={\mu}/{R\mu}$. We then see that
  $\nu$ belongs to $\cK_{\loc}$ and
   the Schr\"odinger form $(\cE^\nu,\cD(\cE^\nu))$ can be defined. In Section 5, we show that  for $\mu\in\cK_H$, 
   $(\cE^\nu,\cD(\cE^\nu))$  is always 
   critical. More precisely, $R\mu$ belongs to the extended Schr\"odinger space 
   $\cD_e(\cE^\nu))$ and $\cE^\nu(R\mu)=0$,
   that is, $R\mu$ is a ground state of $(\cE^\nu,\cD(\cE^\nu))$.  If $\mu$ is of 
   finite energy integral, then $R\mu$ is in 
   the extended Dirichlet space ${\cD_e}(\cE)$; however, it does not hold generally.
    For this reason, it makes sense
   to introduce the extended Schr\"odinger space.

 Denoting $\cL$ the generator of $X$, we have  
  $$
  \cL R\mu+R\mu\cdot \nu=-\mu+\mu=0,
  $$
and a Hardy-type inequality
   \bequ\la{h-in}
  \int_Eu^2d\nu\le \cE(u),\ \ u\in \cD(\cE)\cap C_0(E) 
  \eequ
(Fitzsimmons \cite[Theorem 1.9]{Fi}). Owing to the criticality of $(\cE^\nu,\cD(\cE^\nu))$, we see that 
the inequality (\ref{h-in}) 
 is optimal in the sense that  the inequality (\ref{h-in}) fails to hold if the constant of the left hand side 
  is replaced by a constant bigger than 1. 
  We learn from \cite{P1}, \cite{P2} that the criticality of Schr\"odinger operators
   leads the optimality of  Hardy-type inequalities. For example, if $X$ is the absorbing Brownian motion 
   on $(0,\infty)$ and $\mu(dx)=x^{-3/2}dx$, then $\mu$ belongs to $\cK_H$ and $\nu$ equals
 $(1/4)x^{-2}dx$, which leads us to the classical Hardy inequality. The best constant $1/4$ appears naturally
 by calculating $\mu/R\mu$ and  the constant $3/2$ is determined by the condition $x^{-p}dx\in\cK_H$.
 
In Example \ref{ex5}, we treat the Dirichlet form $(\cE^{(\alpha)},\cD(\cE^{(\alpha)}))$ 
generated a transient symmetric
 $\alpha$-stable process on $\mathbb{R}^d$ ($0 < \alpha < 2,\ \alpha<d$) and consider the Schr\"odinger form 
with Hardy potential: for $u\in\cD(\cE^{(\alpha)}\cap C_0(\mathbb{R}^d)$
\begin{align*}
\cE^{\mu^\delta}(u)&=\cE^{(\alpha)}(u)-\int_{\mathbb{R}^d}u^2 d\mu^\delta,
\end{align*}
where
$\mu^\delta(dx)=\kappa(\delta)/{|x|^\alpha}dx$ and
\bequ\la{kapd}
\kappa(\delta)=\frac{2^{\alpha}\Gamma\left(\frac{\delta+\alpha}{2}\right)
\Gamma\left(\frac{d-\delta}{2}\right)}{\Gamma\left(\frac{\delta}{2}\right)
\Gamma\left(\frac{d-\delta-\alpha}{2}\right)}\ \ (0\le\delta\le d-\alpha).
\eequ
Noting that Cap($\{0\})=0$ and $\mu^\delta$ is in the local Kato class on $\mathbb{R}^d\setminus\{0\}$, 
we see from Theorem \ref{cl}  
 that $(\cE^{\mu^\delta}, C_0^\infty(\mathbb{R}^d))$ 
is closable, Denote by $\cD(\cE^{\mu^\delta})$ and $\cD_e(\cE^{\mu^\delta})$ the closure and 
its extended Schr\"odinger space.
Let $\delta^*=(d-\alpha)/2$ and $\kappa^*=\kappa(\delta^*)$. Then it is known that
\bequ\la{best}
\kappa^*=\frac{2^{\alpha}
\Gamma\left(\frac{d+\alpha}{4}\right)^2}{\Gamma\left(\frac{d-\alpha}{4}\right)^2}
\eequ
is the best constant of the Hardy inequality: for $u\in\cD(\cE^{(\alpha)})$
\bequ\la{hard}
\kappa^*\int_{{\mathbb R}^d}u^2\frac{1}{|x|^\alpha}dx\le \cE^{(\alpha)}(u)
\eequ
(e.g. \cite{BGJP}).
We show in this example that the power $\delta^*$ is determined by the condition $|x|^{-p}dx\in\cK_H$ and 
the best constant $\kappa^*$ comes from the calculation of $\mu^{\delta^*}/R\mu^{\delta^*}$, where 
$R$ is the resolvent kernel of the symmetric
 $\alpha$-stable process.
  We have to say that 
 the argument in this section
 is strongly motivated by that in Miura \cite{Mi}, \cite{Mi1}.

\medskip 
 If $X$ has no killing inside, the semigroup $\widetilde{p}^\nu_t$ defined by
 $$
 \widetilde{p}^\nu_tf(x)=E_x\left(\frac{1}{R\mu(X_0)}\exp\left(\int_0^t\frac{dA^\mu_s}{R\mu(X_s)}\right)
 R\mu(X_t)f(X_t)\right)
 $$
 turns out to be Markovian, $\widetilde{p}^\nu_t1(x)\le 1$ and the Dirichlet form generated by the semigroup 
 $\{\widetilde{p}^\nu_t\}$ equals $(\cE^{\nu,R\mu},\cD(\cE^{\nu,R\mu}))$. The criticality of $(\cE^\nu,\cD(\cE^\nu))$ implies 
 the recurrence of $(\cE^{\nu,R\mu},\cD(\cE^{\nu,R\mu}))$.

\medskip
In Section 6,  we consider again the Dirichlet form $(\cE^{(\alpha)},\cD(\cE^{(\alpha)}))$ and 
 the Schr\"odinger form $(\cE^{\mu^\delta},\cD(\cE^{\mu^\delta}))$ in Example \ref{ex5}.
 We know from \cite[Figure 1]{BGJP} that  
$$
\kappa(\delta)<\kappa^*\ \text{for}\  \delta\not=\delta^*\ \ \text{and}\ \ 
\kappa(\hat{\delta})=\kappa(\delta)\ \text{for}\ \hat{\delta}=d-\alpha-\delta.
$$
Moreover, we see from these facts that 
$$
\lambda(\mu^{\delta^*})=1\ \ \text{and} \ \ \lambda(\mu^\delta)>1\ \text{for}\ \delta\not=\delta^*.
$$
Moreover, it is known in \cite[Theorem 3.1, Theorem 5.4]{BGJP}) that 
the function $|x|^{-\delta},\ 0\le\delta\le\delta^*,$ 
satisfies $p^{\mu^{\delta}}_t|x|^{-\delta}=|x|^{-\delta}$
 and 
the Dirichlet form defined by 
$h$-transform of $\cE^{\mu^\delta}$ by $|x|^{-\delta}$ is written as 
\bequ\la{del}
\cE^{\mu^{\delta},|x|^{-\delta}}(u)=\frac{1}{2} \mathcal{A} (d,\alpha)
\int \!\!\! \int_{\mathbb{R}^d \times \mathbb{R}^d \setminus \triangle }
\frac{(u(x) - u(y))^2}{|x - y|^{d + \alpha}|x|^\delta|y|^\delta} dx dy.
\eequ
In particular, we see that for $0\le\delta\le\delta^*$, 
$|x|^{-\delta}$ belongs to $\cH^+(\mu^\delta)$. 
In this section, we give another proof that 
$(\cE^{\mu^{\delta^*}},\cD(\cE^{\mu^{\delta^*}}))$ is critical in our sense. Indeed, we shall prove that if $\delta=\delta^*$ the corresponding Dirichlet form in (\ref{del}) is recurrent
by showing that 1 is contains 
in the extended Dirichlet space $\cD_e(\cE^{\mu^{\delta^*},|x|^{-\delta^*}})$
 and $\cE^{\mu^{\delta^*},|x|^{-\delta^*}}(1)=0$. 
As a result, we see that the function $|x|^{-\delta^*}$ belongs to the extended Schr\"odinger space
of $\cD_e(\cE^{\mu^{\delta^*}})$ and  $\cE^{\mu^{\delta^*}}(|x|^{-\delta^*})=0$,
that is, $|x|^{-\delta^*}$ is a ground state. Note that $|x|^{-\delta^*}$  
does not belong to the extended Dirichlet space $\cD_e(\cE^{(\alpha)})$. 
Indeed, if $|x|^{-\delta^*}$ belongs to $\cD_e(\cE^{(\alpha)})$,
it must be in $L^2(|x|^{-\alpha} dx)$ by the Hardy inequality (\ref{hard}). 
In other words, if $\delta=\delta^*$, then the extended Schr\"odinger space strictly 
contains the extended Dirichlet space. 
On the other hand, if $\delta\not=\delta^*$, the extended Schr\"odinger space
is identical with $\cD_e(\cE^{(\alpha)})$ (Lemma 3.1),
 which does not contain $|x|^{-\delta}$ by the same reason above. 

We show that for $0\le\delta<\delta^*$ $(\cE^{\mu^\delta},\cD(\cE^{\mu^\delta}))$ is subcritical. As a result,
 we see that $(\cE^{\mu^{\delta},|x|^{-\delta}},\cD(\cE^{\mu^{\delta},|x|^{-\delta}}))$ is transient.
Finally we see that by the {\it inversion property} between $X^\delta$ and $X^{\hat{\delta}}$ 
 $(\cE^{\mu^\delta},\cD(\cE^{\mu^\delta}))$ is transient for $\delta^*<\delta\le d-\alpha$, 
 where $X^\delta$ and $X^{\hat{\delta}}$ are Markov processes generated by the Dirichlet form in (\ref{del}) 
 for $\delta$ and $\hat{\delta}$ respectively (Remark \ref{Inv}).

\section{Schr\"odinger forms}
Let $E$ be a locally compact separable metric space and $m$ a positive
Radon measure on $E$ with full topological support.
Let $(\cE, \cD(\cE))$ be a regular  Dirichlet form on
$L^2(X;m)$. We denote by $u\in \cD_{loc}(\cE)$ if for any relatively compact open set $D$ there exists a function $v\in\cD(\cE)$
such that $u=v$ $m$-a.e. on $D$. 
Let $\cD_e(\cE)$ be the family of $m$-measurable functions $u$ on
$E$ such that $|u| < \infty$ $m$-a.e. and there exists a
sequence $\{ u_{n}\}$ of functions in $\cD(\cE)$ such that $\lim_{n, m \to \infty} \cE(u_n - u_m) = 0$ and 
$\lim_{n \to \infty} u_{n} = u$ $m$-a.e.
We call $\cD_e(\cE)$ the
\emph{extended Dirichlet space} of $(\cE, \cD(\cE))$ and the sequence  $\{ u_{n}\}$ an 
\emph{approximating sequence} of $u\in\cD_e(\cE)$. We can extend the form $\cE$ to $\cD_e(\cE)$: 
for approximating sequences $\{ u_{n}\}$ and $\{v_{n}\}$ of  $u$ and $v\in\cD_e(\cE)$
$$
\cE(u,v)=\lim_{n\to\infty}\cE(u_n,v_n).
$$

\begin{defi}\label{def-transient}
 \begin{enumerate}[(1)]
  \item A Dirichlet form $(\cE, \cD(\cE))$ on $L^2(E;m)$ is said to be \emph{transient}
   if there exists a bounded
	function $g\in L^1(E;m)$ strictly positive $m$-a.e.  such that 
	\begin{equation*}
	 \int_E |u| g dm \le \sqrt{\cE(u)},\ \ u \in \cD(\cE). 
	\end{equation*}
  \item  A Dirichlet form $(\cE, \cD(\cE))$ on $L^2(E;m)$ is said to be \emph{recurrent} if the constant function 1
belongs to $\cD_e(\cE)$ and $\cE(1)=0$.
 \end{enumerate}
\end{defi}
For other characterizations of transience and recurrence, see \cite[Theorem 1.6.2, Theorem 1.6.3]{FOT}.
Throughout this paper, we assume that $(\cE, \cD(\cE))$ is transient and irreducible, i.e., the $L^2(E;m)$-Markov semigroup 
$\{T_t\}$ associated with $(\cE, \cD(\cE))$ satisfies that  
if an $m$-measurable set $A\subset E$ satisfies $T_t(1_Af)=1_AT_tf$ $m$-a.e. for any $f\in L^2(E;m)$ and any $t>0$, then $m(A)=0$ or $m(A^c)=0$.

We define the {\em (1-)capacity\/} $\capa$ associated with the Dirichlet
form $(\cE, \cD(\cE))$ as follows:
for an open set $O\subset E$,
\begin{equation*}
 \capa(O)=\inf\{\cE_{1}(u)\mid u\in \cD(\cE), u\geq 1,\ m\text{-a.e. on
  }O\}
\end{equation*}
and for a Borel set $A\subset E$,
\begin{equation*}
 \capa(A)=\inf\{\capa(O)\mid O\text{ is open, }O\supset  A\},
\end{equation*}
where $\cE_1 (u)=\cE(u)+(u,u)_m$. 

A statement depending on $x\in E$ is said to hold q.e. on $E$ if there
exists a set $N\subset E$ of zero capacity such that the statement is
true for every $x\in E\setminus N$.
``q.e.'' is an abbreviation of ``quasi-everywhere''.
A real valued function $u$ defined q.e. on $E$
is said to be {\em quasi-continuous\/} if for any
$\epsilon>0$ there exists 
an open set $G\subset E$ such that $\capa(G)<\epsilon$ and
$u|_{E\setminus G}$ is finite and  continuous.
Here, $u|_{E\setminus G}$ denotes the restriction of $u$ to $E\setminus
G$.
Each function $u$ in $\cD_e(\cE)$ admits a
quasi-continuous version $\tilde{u}$, that is,
$u=\tilde{u}$ $m$-a.e. 
In the sequel,
we always assume that every function $u\in \cD_e(\cE)$ is
represented by its quasi-continuous version.

\smallskip
We call  a positive Borel measure $\mu$ on $E$  {\it smooth} if it satisfies 

 \begin{enumerate}[(i)]
  \item $\mu$ charges no set of zero capacity,
\item  there exists an increasing sequence $\{F_n\}$ of closed sets such that 

\smallskip
a) $\mu(F_n)<\infty,\ n=1,2,\ldots,$

b) $\lim_{n\to\infty}{\rm Cap}(K\setminus F_n)=0$ for any compact set $K$.
 \end{enumerate}

Let $\mu$ be a positive smooth Radon measure. 
 We define a Schr\"odinger form on $L^2(E;m)$ by
\begin{equation}\la{sch-form}
\cE^{\mu} (u,v) = \cE(u,v) - \int_E uv d\mu,\ \ u,v\in  \cD(\cE)\cap C_0(E).
\end{equation}
We assume that $(\cE^{\mu},\cD(\cE)\cap C_0(E))$ is positive semi-definite, $\cE^\mu(u)\ge 0$ 
for $u\in\cD(\cE)\cap C_0(E)$, and closable. We denote by
 $(\cE^{\mu},\cD(\cE^\mu))$  its closure. By the regularity of $(\cE,\cD(\cE))$, $\cD(\cE)\subset \cD(\cE^\mu)$
 and 
 $$
 \int_E{u}^2d\mu\le\cE(u),\ \ \cE^\mu(u)=\cE(u)-\int_E{u}^2d\mu,\ \ u\in\cD(\cE).
 $$
The Schr\"odinger form $(\cE^\mu,\cD(\cE^\mu))$ is expressed by the non-positive self-adjoint operator
 $\cH^\mu$ as follows: 
$$
\cE^\mu(u)=\|\sqrt{-\cH^\mu}u\|_2^2,\ \ \cD(\cE^\mu)=\cD(\sqrt{-\cH^\mu}).
$$ 
 Then the $L^2(E;m)$-strong continuous semigroup $T^\mu_t:=\exp(t\cH^\mu)$ 
 is defined and it is contractive, $\|T^\mu_t\|_2\le 1$, where $\|\cdot\|_2$ is the operator 
norm on $L^2(E;m)$. 

A densely defined, closed, positive semi-definite symmetric bilinear form $(a,\cD(a))$ is said to be 
{\it positive preserving}  if  for $u\in\cD(a)$ $|u|$ belongs to $\cD(a)$ and $a(|u|)\leq a(u)$. 
It follow from \cite[Lemma 1.3.4]{Da} that the form $(\cE^\mu,\cD(\cE^\mu))$ is 
	{positive preserving} because $\cE^\mu(|u|)\leq \cE^\mu(u)$ for $u\in\cD(\cE)\cap C_0(E)$.
	As a result,  we see from \cite[Proposition 2]{Sch1} that 
	$(\cE^\mu,\cD(\cE^\mu))$ has the {\it Fatou property}, i.e., if $\{u_n\}\subset\cD(\cE^\mu)$
	 satisfies $\sup_n\cE^\mu(u_n)<\infty$ and $u_n\to u\in\cD(\cE^\mu)$ $m$-a.e., then 
	$\liminf_{n\to\infty}\cE^\mu(u_n)\geq\cE^\mu(u)$.
Hence, following \cite{Sch}, we can define a space ${\cD_e}(\cE^{\mu})$ in the way similar to the extended 
Dirichlet space: An $m$-measurable function $u$
 with $|u| < \infty$ $m$-a.e. is said to be in ${\cD_e}(\cE^{\mu})$ if there exists an $\cE^\mu$-Cauchy
sequence $\{ u_{n}\}\subset\cD(\cE^\mu)$ such that $\lim_{n \to \infty} u_{n} = u$ $m$-a.e. 
We call ${\cD_e}(\cE^{\mu})$ the {\it extended Schr\"{o}dinger  space}  
of $(\cE^\mu,\cD(\cE^\mu))$ and  the sequence $\{u_n\}$ an {\it approximating sequence} of $u$. For 
$u\in {\cD_e} (\cE^{\mu})$ and an approximating sequence $\{u_n\}$ of $u$ define
\begin{equation}
{\cE}^{\mu}(u)=\lim_{n\to\infty} \cE^{\mu}(u_n).
\end{equation}

We can give another definition of the extended Schr\"{o}dinger space through $h$-transform. 
We introduce the space of $T^\mu_t$-excessive functions:
	\begin{equation}\la{mu-exc}
		\cH^+(\mu) = \{ h \,\mid\,  0<h<\infty\ m\text{-a.e.},\   T_t^{\mu} h \le h\  m\text{-a.e.}\}.
\end{equation}		
For $h\in\cH^+(\mu)$ define the Dirichlet form $(\cE^{\mu,h}, \cD(\cE^{\mu,h}))$ by
\begin{equation}\la{h-dir}
\left\{
\begin{split}
& \cE^{\mu,h}(u,v) = \cE^{\mu}(uh,uh),\\
& \cD(\cE^{\mu,h}) = \left\{ u\in L^2(E;h^2m) \mid uh \in \cD(\cE^{\mu})
\right\},
\end{split}
\right.
\end{equation}
and let $\cD_e(\cE^{\mu,h})$ be the extended Dirichlet space of $(\cE^{\mu,h}, \cD(\cE^{\mu,h}))$. 
We then see that $u\in{\cD_e} (\cE^{\mu})$ is equivalent to $u/h\in{\cD_e} (\cE^{\mu,h})$
(\cite[Lemma 2.8]{T3}).

We define the criticality 
and subcriticality of Schr\"{o}dinger forms in the way similar to the  
recurrence and transience of Dirichlet forms. 

\begin{defi}\label{def-criticality} Let $\mu$ be a smooth Radon measure and 
$(\cE^{\mu}, \cD(\cE^{\mu}))$ the positive semi-definite Schr\"odinger form.
\begin{enumerate}[(1)]
\item $(\cE^{\mu}, \cD(\cE^{\mu}))$ is said to
	be \emph{subcritical} if there exists a bounded function $g$ in $L^1(E;m)$ strictly positive $m$-a.e. such that 
\bequ\la{tra}
	\int_E|u| g dm\le \sqrt{\cE^\mu(u)},\ \ u\in\cD_e(\cE^{\mu}).
\eequ
\item $(\cE^{\mu}, \cD(\cE^{\mu}))$ is
	 said to be \emph{critical} if there exists a function $\phi$ in $\cD_e(\cE^{\mu})$ strictly positive
	  $m$-a.e. such that $\cE^\mu(\phi)=0$. The function $\phi$ is said to be the {\it ground state}.
\end{enumerate}
\end{defi}

Define the operator $G^\mu$ on a positive function $f$ by
$$
G^\mu f(x)=\int_0^\infty T^\mu_tf(x)dt\ (\le +\infty).
$$

\begin{lem}\la{77}
Let $g$ be the function in Definition \ref{def-criticality} (1). Then $G^\mu g$ belongs to $\cD_e(\cE^\mu)$.
\end{lem}
\begin{proof}
By the same argument as in \cite[Lemma 1.5.3]{FOT} we have the next inequality: 
for any non-negative $f\in L^1(E;m)\cap L^2(E;m)$
$$
\sup_{u\in\cD(\cE^\mu)}\frac{(|u|,f)}{\sqrt{\cE^\mu(u)}}=\sqrt{\int_Ef G^\mu f\, dm}\ (\le +\infty).
$$
Hence, the equation (\ref{tra}) implies $\int_Eg G^\mu g\, dm\le 1$. Moreover, by the same argument as in Theorem 1.5.4 (i) in \cite{FOT} we see 
that $G^\mu g$ belongs to $\cD_e(\cE^\mu)$.
\end{proof}

\begin{lem}\la{rc}
Suppose $(\cE^\mu,\cD(\cE^\mu))$ is critical and let $\phi$ be the ground state. 
Then $\phi$ belongs to $\cH^+(\mu)$, more precisely, $T^\mu_t\phi=\phi$ $m$-a.e.
\end{lem}
\begin{proof}
We see from Lemma \ref{ex-ine} below.
\end{proof}


\begin{rem}  \la{g-u} If $h\in\cD_e(\cE^\mu)$ satisfies $\cE^\mu(h)=0$, then $h$ is a weak 
solution to $\cH^\mu h=0$ in the sense that $\cE^\mu(h, \varphi)=0$ for any 
$\varphi\in \cD(\cE)\cap C_0(E)$ because
 $(\cE^{\mu}, \cD(\cE^{\mu}))$ is positive semi-definite. Moreover, if $(\cE^\mu,\cD(\cE^\mu))$ is critical and 
 $\phi$ is a ground state, then the function $h$ is written as $h=c\phi$ ($c$ is a constant), in particular, 
 not sign-changing. Indeed, by Theorem \ref{tai}, the Dirichlet form $(\cE^{\mu,\phi},\cD(\cE^{\mu,\phi}))$
  is recurrent and
$$
\cE^{\mu,\phi}(h/\phi)=\cE^{\mu}(h)=0,
$$
which implies $h/\phi$ is a constant $m$-a.e. by \cite[Theorem 1]{Ka}. 
\end{rem}

\begin{lem}\la{fo1}
If $u\in\cD_{e}(\cE^\mu)$ and $\{u_n\}_{n=1}^\infty$ is 
 an approximating sequence of $u$, then $\lim_{n\to\infty}\cE^\mu(u-u_n)=0$.
\end{lem}

\begin{proof}
For any $\varepsilon>0$ there exists $N$ such that 
$\cE^\mu(u_m-u_n)	\leq\varepsilon$ for $m,n\geq N$. 
For a fixed $n$, $\{u_m-u_n\}_{m=1}^\infty$ is an 
approximating sequence of $u-u_n$ and so 
$$
\cE^\mu(u-u_n)=\lim_{m\to\infty}\cE^\mu(u_m-u_n)
\leq\varepsilon,\ \ n\geq N.
$$
\end{proof}

\begin{lem}\la{ex-h}
If $(\cE^\mu,\cD(\cE^\mu))$ is subcritical, then $(\cD_e(\cE^\mu),\cE^\mu)$ is a Hilbert \\
space.
\end{lem}

\begin{proof}
Let $\{u_n\}$ be a Cauchy sequence with respect to 
$\cE^\mu$.   Then $\{u_n\}$ is also a Cauchy sequence in 
$L^1(E;gm)$ by (\ref{tra}) and so there exists a subsequence 
 $\{u_k\}$ of $\{u_n\}$ converges to 
$u\in L^1(E;gm)$ $m$-a.e. Hence $u$ belongs to $\cD_e(\cE^\mu)$ and 
$\{u_k\}$ converges to $u$ in $\cE^\mu$ by Lemma \ref{fo1}. Since $\{u_n\}$
 be a Cauchy sequence in $\cE^\mu$, $\{u_n\}$ itself 
 converges to $u$ in $\cE^\mu$.
\end{proof}

We can show in the same argument 
as in \cite[Lemma 3.13]{T4} that the semigroup 
$T^\mu_t$, $t>0$ can be extended to an operator 
form $\cD_e(\cE^\mu)$ to itself. 
For $u\in\cD(\cE^\mu)$ 
\begin{align*}
\frac{1}{t}\|u-T_t^\mu u\|_2^2
=\frac{1}{t}(u-T_t^\mu u,u)
-\frac{1}{t}(u-T_t^\mu u,T_t^\mu u)\leq\frac{1}{t}(u-T_t^\mu u,u)
\end{align*}
because 
$$
(u-T_t^\mu u,T_t^\mu u)=\int_0^\infty\left( 1-e^{-\lambda t}\right)
e^{-\lambda t}d(E_\lambda u,u)\geq 0 
$$
by the spectral decomposition theorem, $-H^\mu=\int_0^\infty\lambda dE_{\lambda}$. Here 
$\{E_\lambda\}_{\lambda\ge 0}$ is a resolution of the identity. Hence we have

$$
\frac{1}{t}\|u-T_t^\mu u\|_2^2\le\cE^\mu(u).
$$

Since
\begin{align*}
\frac{1}{t}(u-T_t^\mu u,u)&
=\int_0^\infty\left(\frac{1-e^{-\lambda t}}
{t}\right)d(E_\lambda u,u)\\
&\le \int_0^\infty\lambda d(E_\lambda u,u)\\
&=\cE^\mu(u).
\end{align*}

For $u\in\cD_e(\cE^\mu)$, let $\{u_n\}_{n=1}^\infty\subset \cD(\cE^\mu) $ 
be an approximating sequence of $u$. Then 
\begin{align*}
\frac{1}{t}\|(u_n-T_t^\mu u_n)-
(u_m-T_t^\mu u_m)\|_2^2&
=\frac{1}{t}\|(u_n-u_m)-T^\mu_t(u_n-u_m)\|_2^2\\
&\leq \cE^\mu(u_n-u_m),
\end{align*}
and thus $(u_n-T_t^\mu u_n)$
 converges to a $v\in L^2(E;m)$ strongly. 
 Let $\{u'_n\}_{n=1}^\infty\subset \cD(\cE^\mu) $ 
 be another 
approximating sequence of $u$. Then 
\begin{align*}
&\frac{1}{t}\|(u_n-T_t^\mu u_n)
-(u'_n-T_t^\mu u'_n)\|_2^2\leq \cE^\mu(u_n-u'_n)\\
&\qquad = \cE^\mu(u_n)+\cE^\mu(u'_n)
-2\cE^\mu(u_n,u'_n)\\
&\qquad \ \ \longrightarrow \cE^\mu(u)
+\cE^\mu(u)-2\cE^\mu(u)=0
\end{align*}
as $n\to \infty$. Hence the function $v\in L^2(E;m)$
 is independent of the choice of the approximating sequence.

Using the spectral decomposition theorem again, we have for $u\in\cD(\cE^\mu)$
$$
\cE^\mu(T_t^\mu u)=\int_0^\infty\lambda\left( e^{-2\lambda t}\right)^2 d(E_\lambda u,u)\leq
\int_0^\infty\lambda d(E_\lambda u,u)\leq\cE^\mu(u),
$$
and $\cE^\mu(T_t^\mu u_n
-T_t^\mu u_m)\leq \cE^\mu(u_n-u_m)$. 
There exists a subsequence $\{u_{n_k}\}_{k=1}^\infty$ of
 $\{u_n\}_{n=1}^\infty$ such that $T_t^\mu 
 u_{n_k}=u_{n_k}-(u_{n_k}-T_t^\mu u_{n_k})$ 
 converges to $u-v$ $m$-a.e. Hence, $u-v\in\cD_e(\cE^\mu)$
  and $\{T_t^\mu u_{n_k}\}_{k=1}^\infty$
   is an approximating sequence of $u-v$.
The semigroup $T_t^\mu$ can be extended to
 $\cD_e(\cE^\mu)$ by $T_t^\mu u=u-v$. 
 We then see that 
$\lim_{n\to\infty}\|(u_n-T_t^\mu u_n)-(u-T_t^\mu u)\|_2=0$ and 
\begin{align*}
\frac{1}{t}\|u-T_t^\mu u\|_2^2&=\frac{1}{t}\|v\|_2^2=\lim_{n\to\infty}\frac{1}{t}\|u_{n}
-T_t^\mu u_{n}\|_2^2\\
&\leq\lim_{n\to\infty}\cE^\mu(u_{n})=\cE^\mu(u).
\end{align*}
Hence, we have
\begin{lem}\la{ex-ine}
The semigroup $T_t^\mu$ can be uniquely extended to a linear 
operator on $\cD_e(\cE^\mu)$ and for $u\in\cD_e(\cE^\mu)$
	$$
	\frac{1}{t}\|u-T_t^\mu u\|_2^2\leq \cE^\mu(u).
	$$
\end{lem}

The argument similar to that in the proof of \cite[Lemma 1.5.4]{FOT}  leads us to the next lemma.	
\begin{lem}\la{ex-ine1}
For	$u\in\cD_e(\cE^\mu)$ and $w\in\cD(\cE^\mu)$
	$$
	\lim\limits_{t\downarrow 0}\frac{1}{t}(u-T_t^\mu u,w)=\cE^\mu(u,w).
	$$
\end{lem}
\begin{proof}
For $u,v\in\cD(\cE^\mu)$
\begin{align}\la{abo1}
	\left|\frac{1}{t}(u-T_t^\mu u,w)\right|&\leq \left( \frac{1}{t}(u-T_t^\mu u,u)\right) ^{1/2}
	\left( \frac{1}{t}(w-T_t^\mu w,w)\right) ^{1/2}\\
	&\leq \cE^\mu(u)^{1/2}\cE^\mu(w)^{1/2}.\nonumber
\end{align}

Let $\{u_n\}_{n=1}^\infty\subset \cD(\cE^\mu) $ be an approximating sequence of $u$. 
Noting that the inequality (\ref{abo1}) can be extended to $u\in\cD_e(\cE^\mu)$ by the argument 
before Lemma \ref{ex-ine}, we have
\begin{align*}
	\left|\frac{1}{t}(u-T_t^\mu u,w)-\frac{1}{t}(u_n-T_t^\mu u_n,w)\right|
	&=\left|\frac{1}{t}(u-u_n-T_t^\mu(u-u_n),w)\right|\\
	&\leq \cE^\mu(u-u_n)^{1/2}\cE^\mu(w)^{1/2},
\end{align*}
equivalently
\begin{align*}
&\frac{1}{t}(u_n-T_t^\mu u_n,w)-\cE^\mu(u-u_n)^{1/2}\cE^\mu(w)^{1/2}
\leq \frac{1}{t}(u-T_t^\mu u,w)\\
&\qquad \qquad \qquad \qquad \leq\frac{1}{t}(u_n-T_t^\mu u_n ,w)+\cE^\mu(u-u_n)^{1/2}\cE^\mu(w)^{1/2}.
\end{align*}
Hence we have  
	\begin{align}\la{1ub}
	\cE^\mu(u_n,w)-\cE^\mu(u-u_n)^{1/2}\cE^\mu(w)^{1/2}&\leq 
	\varliminf_{t\downarrow 0}\frac{1}{t}(u-T_t^\mu u,w)\nonumber \\
	&\leq \varlimsup_{t\downarrow 0}\frac{1}{t}(u-T_t^\mu u,w)\\
	&\leq \cE^\mu(u_n,w)+\cE^\mu(u-u_n)^{1/2}\cE^\mu(w)^{1/2}.\nonumber
	\end{align}
The both sides of (\ref{1ub}) tend to $\cE^\mu(u,w)$ 
as $n\to\infty $ and the proof is completed.
\end{proof}

\begin{lem}\la{well1}
Suppose $\cH^+(\mu)$ is not empty. If $(\cE^{\mu,h}, \cD(\cE^{\mu,h}))$ is
transient for some $h\in\cH^+(\mu)$,
then so is for any $h\in\cH^+(\mu)$. 
\end{lem}

\begin{proof}
First note that by the irreducibility of $(\cE,\cD(\cE))$,   
 $(\cE^{\mu,h}, \cD(\cE^{\mu,h}))$ is either transient or recurrent. 
Suppose that for 
$h_1, h_2\in \cH^+(\mu)$, $(\cE^{\mu,h_1}, \cD(\cE^{\mu,h_1}))$ is transient and 
$(\cE^{\mu,h_2}, \cD(\cE^{\mu,h_2}))$ is recurrent.
 Then it follows that there exists $g > 0$  such that 
 \begin{equation}\label{r-2-3-1}
 \int_E |u|g h_1^2 dm \le \sqrt{\cE^{\mu,h_1}(u)}, \ \ 
 u \in \cD_e(\cE^{\mu,h_1}).
\end{equation}
Let $\{\psi_n\}\subset\cD(\cE^{\mu,h_2})$ be an approximating sequence 1 and 
$\cE^{\mu,h_2}(\psi_n)\to 0$ as $n\to\infty$.
Noting that $(h_2/h_1) \psi_n=(h_2\psi_n)/h_1\in \cD_e(\cE^{\mu,h_1})$, we have 
\begin{align}\label{r-2-3-1}
 \int_E ({h_2}/{h_1})|\psi_n|g h_1^2 dm &\le \sqrt{\cE^{\mu,h_1}((h_2/h_1) \psi_n))}
= \sqrt{\cE^{\mu}(h_2\psi_n)}\\
&=\sqrt{\cE^{\mu,h_2}(\psi_n)}\to 0,\nonumber
\end{align}
which is contradictory because 
$$
\liminf_{n\to\infty} \int_E ({h_2}/{h_1})|\psi_n|g h_1^2 dm\ge\int_Eh_1h_2g dm>0.
$$
\end{proof}

\begin{lem}\la{ma}
If $(\cE^{\mu}, \cD(\cE^{\mu}))$ is subcritical, then the function $G^\mu g$ in Lemma \ref{77} belongs to
 $\cH^+(\mu)$ and $(\cE^{\mu,G^\mu g}, \cD(\cE^{\mu,G^\mu g}))$
 is transient.
\end{lem}
\begin{proof}
The function $G^\mu g\in \cD_e(\cE^\mu)$ belongs to $\cH^+(\mu)$. Indeed, let
$$
S^\mu_Tg(x)=\int_0^TT^\mu_sg ds.
$$
Then 
$$
T^\mu_tS^\mu_Tg(x)=\int_0^{T+t}T^\mu_sg ds-\int_0^tT^\mu_sg ds
$$
and by $T\to\infty$
\bequ\la{ex}
T^\mu_tG^\mu g(x)=G^\mu g -\int_0^tT^\mu_sg ds\le G^\mu g .
\eequ
Put $v=G^\mu g$. Then since 
$$
\cE^{\mu,v}(1)=\cE^\mu(v)=\int_E g G^\mu g dm >0,
$$
$(\cE^{\mu,v},\cD(\cE^{\mu,v}))$ is transient, more precisely, has the killing part. 
\end{proof}

 \begin{thm}\la{tai}
{\rm i)}  $(\cE^{\mu}, \cD(\cE^{\mu}))$ is {subcritical} if and only if there exists $h\in\cH^+(\mu)$ such that 
$(\cE^{\mu,h},\cD(\cE^{\mu,h}))$ is transient.\\
{\rm ii)}  $(\cE^{\mu}, \cD(\cE^{\mu}))$ is critical if and only if there exists $h\in\cH^+(\mu)$ such that 
$(\cE^{\mu,h},\cD(\cE^{\mu,h}))$ is recurrent.
 \end{thm}
 \begin{proof}
 i)  The proof of ``only if'' part is given in Lemma \ref{ma}.
 Let $h$ be a function in $\cH^+(\mu)$ such that $(\cE^{\mu,h},\cD(\cE^{\mu,h}))$ is transient. Then 
 there exists a bounded function $g'\in L^1(E;h^2m)$ strictly positive $m$-a.e. such that for $u\in\cD_e(\cE^{\mu,h})$
	$$
	\int_E|u| g' h^2dm\le \sqrt{\cE^{\mu,h}(u)}.
	$$
Putting $v=hu$, we have for $v\in\cD_e(\cE^{\mu})$
$$
\int_E|v| g' hdm\le \sqrt{\cE^{\mu}(v)}.
$$	
Let $\eta\in L^1(E;m)$ with $0<\eta \le 1,	\ m$-a.e. and define $g=((g' h)\wedge 1)\eta$.
Then $g$ is a bounded  function in $L^1(E;m)$ positive $m$-a.e. such that
$$
\int_E|v| g dm\le \sqrt{\cE^{\mu}(v)},\ \ \forall v\in\cD_e(\cE^{\mu}).
$$	

ii) Let $\phi$ be a function in  Definition \ref{def-criticality} (2). 
Then we see from Lemma \ref{rc} that $\phi\in\cD_e(\cE^\mu)\cap \cH^+(\mu)$. Hence
$1\in \cD_e(\cE^{\mu,\phi})$ and 
$$
\cE^{\mu,\phi}(1)=\cE^\mu(\phi)=0.
$$

Suppose that for  an $h\in \cH^+(\mu)$ $(\cE^{\mu,h},\cD(\cE^{\mu,h}))$ is recurrent. 
Then $1\in \cD_e(\cE^{\mu,h})$ and 
$\cE^{\mu,h}(1)=0$. Hence $h$ is in $\cD_e(\cE^{\mu})$ and satisfies $\cE^{\mu}(h)=0$, that is, a ground state
by Remark \ref{g-u}.
 \end{proof}

\section{Analytic Criterion for Subcriticality}
We define a function space
\bequ
\cL=\left\{f\, \Big|\,  \Big|\int_E |f|\varphi 
dm\Big|\le C\cE(\varphi)^{1/2}\ \text{for any}\ \varphi\in\cD_e(\cE) \right\}.
\eequ
In this section, we make an assumption:
\bequ\la{BD} 
1_K\in\cL\ \  \text{for any compact set}\ K.
\eequ
Note that $\cD(\cE)\cap C_0(E)\subset\cL$ by the assumption
If the Hunt process $X$ generated by $(\cE,\cD(\cE))$ satisfies the strong Feller property, 
then the assumption above is fulfilled
by the Green-boundedness of $1_Km$, $\|G1_K\|_\infty<\infty$ (\cite[Proposition 2.2]{Chen-gauge}) and
the inequality 
\begin{align}\la{domi}
\int_E |\psi|\varphi dm\le \|G|\psi|\|_\infty^{1/2}\cdot\left(\int_E|\psi|dm\right)^{1/2}\cE(\varphi)^{1/2}.
\end{align}
We define 
\begin{equation}\la{j-ground}
\lambda (\mu)= \inf \left\{ \cE(u)\,\Big|\, u \in\cD(\cE)\cap C_0(E),\ 
		        \int_E u^2 d\mu = 1\right\}.
\end{equation}

\begin{lem}\la{e-1}
If $\lambda (\mu)>1$, then there exists a positive constant $c$ such that 
$$
c\cdot\cE(u)\le \cE^\mu(u)\le \cE(u), \ u\in \cD(\cE)\cap C_0(E).
$$
\end{lem}
\begin{proof}
By the definition of $\lambda(\mu)$
$$
\int_Eu^2d\mu\le \frac{1}{\lambda(\mu)}\cdot\cE(u),\ u\in\cD(\cE)\cap C_0(E)
$$
and thus
$$
\cE^\mu(u)=\cE(u)-\int_Eu^2d\mu\ge \cE(u)-\frac{1}{\lambda(\mu)}\cE(u)=\left(1-\frac{1}{\lambda(\mu)}\right)\cE(u).
$$
\end{proof}
We see from Lemma \ref{e-1} that if $\lambda (\mu)>1$, then $(\cE^\mu,\cD(\cE)\cap C_0(E))$ is 
positive semi-definite and closable. Furthermore, its extended Schr\"{o}dinger space $\cD_e(\cE^\mu)$
 equals the extended Dirichlet space $\cD_e(\cE)$ and $(\cD_e(\cE^\mu), \cE^\mu)$ is a Hilbert space.

\begin{lem}\la{e-2}
Suppose $\lambda (\mu)>1$. For $f\in \cL$, there exists an $h\in\cD_e(\cE^\mu)$ such that for any $\varphi\in\cD_e(\cE^\mu)$
$$
\cE^\mu(h,\varphi)=\int_Ef\varphi dm.
$$
\end{lem}
\begin{proof}
Since $(\cD_e(\cE^\mu),\cE^\mu)$ is a Hilbert space, the Riesz theorem leads us to this lemma.  
\end{proof}

\begin{lem}\la{e-3}
Suppose $\lambda (\mu)>1$. If $f\in \cL$ is non-negative, then $h\in\cD_e(\cE^\mu)$ in Lemma
\ref{e-2} is also non-negative. 
\end{lem}
\begin{proof}
Since $\cE^\mu(h,\psi)=\int_Ef\psi dm\ge 0$
 for any non-negative $\psi\in\cD_e(\cE^\mu)$,  
 $\cE^\mu(h,h^-)\ge 0$ and so 
 $\cE^\mu(h^+,h^-)\ge \cE^\mu(h^-)$. By the positive preserving property of $\cE^\mu$
 $$
 0\le \cE^\mu(h^-)\le\cE^\mu(h^+,h^-)\le 0,
 $$
and thus $\cE^\mu(h^-)=0$. Hence $\cE(h^-)=0$ by Lemma \ref{e-1} 
and $h^-=0$ by the transience of $(\cE,\cD(\cE))$.
\end{proof}

We showed in some paragraphs before Lemma \ref{ex-ine} that the semigroup $T^\mu_t$
 can be extended to a operator on $\cD_e(\cE^\mu)$. By the same argument for 
 $T^\mu_t$, the resolvent $G^\mu_\alpha$, $\alpha>0$ can also be extended to a operator on $\cD_e(\cE^\mu)$.

\begin{lem}\la{e-4}
Suppose $\lambda (\mu)>1$. Let $f$ and $h$ be functions in Lemma \ref{e-2}. Then 
$$
h-\alpha G^\mu_\alpha h=G^\mu_\alpha f.
$$
\end{lem}
\begin{proof}
We can construct an approximating sequence $\{h_n\}\subset \cD(\cE^\mu)$ of 
$h$ such that $\{G^\mu_\alpha h_n\}$ is an approximation sequence of $G_\alpha^\mu h$. 
Hence, we have for $\psi\in\cD(\cE)\cap C_0(E)$
  \begin{align*}
\cE^\mu_\alpha(G^\mu_\alpha h,\psi)=\lim_{n\to\infty}\cE^\mu_\alpha(G^\mu_\alpha h_n,\psi)
=\lim_{n\to\infty}(h_n,\psi)=(h,\psi)_m
 \end{align*}
because $\psi\in\cL$ and  
\begin{align}\la{domi}
\int_E(h-h_n)\psi dm&\le C\cE(h-h_n)^{1/2}.
\end{align}
Hence
 \begin{align*}
\cE^\mu_\alpha(h-\alpha G^\mu_\alpha h,\psi)&=\int_Ef\psi dm -\alpha\int_Eh\psi dm+\alpha\int_Eh\psi dm\\
&=\int_Ef\psi dm=\cE^\mu_\alpha(G^\mu_\alpha f,\psi),\ \forall \psi\in\cD(\cE)\cap C_0(E).
 \end{align*}
\end{proof}

\begin{rem}
We see from (\ref{domi}) that $\psi\in\cD(\cE)\cap C_0(E)$ belongs to the space $\cL$.
\end{rem}

\begin{lem}\la{e-5}
For $f\in\cL$, the function $G^\mu f$ belongs to $\cD_e(\cE^\mu)$ and
equals $h$ defined in Lemma \ref{e-2}. 
\end{lem}
\begin{proof}
We may suppose $f$ is non-negative, $f\ge 0$. Noting that $G^\mu_\alpha f\le h$ by Lemma \ref{e-4}, we have 
 \begin{align*}
\cE^\mu(G^\mu_\alpha f)&\le\cE^\mu_\alpha(G^\mu_\alpha f)=\int_E fG^\mu_\alpha f dm \le \int_Eh f dm<\infty,
 \end{align*}
and so $\sup_{\alpha>0}\cE^\mu(G^\mu_\alpha f)<\infty.$ Since 
$G^\mu_\alpha f\uparrow G^\mu f$ as $\alpha\to 0$, we see from Banach-Alaoglu theorem that 
for a certain sequence $\alpha_n\downarrow 0$, $G^\mu_{\alpha_n} f$ converges $\cE^\mu$-weakly to $G^\mu f\in\cD_e(\cE^\mu)$.

Since 
$$
|\alpha_n(G^\mu_{\alpha_n} f,\psi)_m|\le \alpha_n(h,|\psi|)_m\to 0,\ \ n\to \infty,
$$
we have 
$$
\cE^\mu(h,\psi)=\int_E f\psi dm=\cE_{\alpha_n}^\mu(G^\mu_{\alpha_n} f,\psi)\to \cE^\mu(G^\mu f,\psi), \ \ n\to \infty.
$$
Hence, $\cE^\mu(h,\psi)=\cE^\mu(G^\mu f,\psi)$ for any $\psi\in\cD_e(\cE^\mu)$, and $h=G^\mu f$.
\end{proof}

The fact in Lemma \ref{e-5} above is proved in \cite{BBGM} for Dirichlet forms generated by 
rotationally symmetric $\alpha$-stable processes.

\begin{thm}\la{ma2}
If $\lambda (\mu)>1$, then $(\cE^{\mu}, \cD(\cE^{\mu}))$ is \emph{subcritical}.
\end{thm}
\begin{proof}
For a non-negative function $f\in C_0(E)$ with $f\not\equiv 0$, $h=G^\mu f$ belongs to $\cH^+(\mu)$ because
$$
T^\mu_tG^\mu f(x)\le G^\mu f(x)
$$
by lemma \ref{ex}. Since
$$
\cE^{\mu,h}(1)=\cE^\mu(h)=\int_Ef G^\mu f dm >0,
$$
$(\cE^{\mu,h},\cD(\cE^{\mu,h}))$ is transient, more precisely, has the killing part. 
\end{proof}

\begin{rem} 
As stated in Introduction, for a general 
positive smooth Radon measure with $\lambda(\mu)=1$ we cannot construct the ground state 
of $(\cE^\mu,\cD(\cE^\mu))$, and we do not know whether $(\cE^\mu,\cD(\cE^\mu))$ is critical or not.
\end{rem}

\section{Probabilistic representation of Schr\"{o}dinger semigroups}
 In this section, we give a sufficient condition for $\mu$ that $(\cE^\mu,\cD(\cE)\cap C_0(E))$ is closable
 and its Schr\"{o}dinger semigroup $T^\mu_t$ can be expressed by a Feynman-Kac semigroup (\ref{f-k}).

 Let $X= (\Om, \cF, \{\cF_{t}\}_{t \ge 0}, \{P_x\}_{x \in E}, \{X_t\}_{t \ge 0},
\zeta)$ be the symmetric Hunt process generated by $(\cE, \cD(\cE))$,
where $\{ \cF_t\}_{t \ge 0}$ is the augmented filtration and $\zeta $ is the lifetime of $X$.
 Denote by $\{p_t\}_{t\geq 0}$ and  $\{R_\alpha\}_{\alpha \geq 0} $ the semigroup and resolvent 
of $X$:
$$
p_tf(x)=E_x(f(X_t)),\ \ \ \ \ R_\alpha f(x)=\int_0^\infty e^{-\alpha t} p_tf(x)dt.
$$
Then $p_tf(x)=T_tf(x)$ $m$-a.e., $R_\alpha f(x)=\int_0^\infty T_tf(x)dt$ $m$-a.e.
Let us remember that $(\cE,\cD(\cE))$ is supposed to be irreducible and transient throughout this paper. 
Consequently, the corresponding Markov process $X$ 
is irreducible and transient. 
In the sequel, we assume that $X$ satisfies, in addition, the next condition:  

\medskip
{\textbf{Strong Feller Property (SF)}}. \  For each $t$,
      $p_t ({\mathscr  B}_b(E))\subset C_b(E)$, 
      where $C_b(E)$ is the space of bounded continuous functions on $E$.
\medskip

We remark that \textbf{(SF)} implies 

\medskip
{\textbf{Absolute Continuity Condition (AC)}}. \ The transition probability of $X$ is absolutely continuous with respect to $m$, $p(t,x,dy)=p(t,x,y)m(dy)$ for each $t>0$ and $x\in E$.

\medskip
Under {\textbf{(AC)}}, there exists a non-negative, jointly measurable $\alpha$-resolvent kernel $R_{\alpha}(x,y)$: 
For $x\in E$ and $f\in{\mathscr B}_b(E)$
$$
R_\alpha f(x)=\int_ER_\alpha (x,y)f(y)m(dy).
$$
Moreover, $R_\alpha (x,y)$ is $\alpha $-excessive in $x$ and in $y$ (\cite[Lemma 4.2.4]{FOT}).
We  simply write $R(x,y)$ for $R_{0}(x,y)$. 
For a measure $\mu$, we define the $\alpha $-potential of $\mu$ by
\begin{equation*}
R_{\alpha} \mu(x) = \int_{E} R_{\alpha}(x,y) \mu(dy).
\end{equation*}

Let $S_{00}$ be the set of positive Borel measures $\mu$ such that $\mu(E)<\infty $ and $R_1\mu$ is bounded.
We call a Borel measure $\mu$ on $E$ {\it smooth} if there exists a sequence $\{E_n\}$ of Borel sets increasing to 
$E$ such that for each $n$ $1_{E_n}\cdot \mu\in S_{00}$ and for any $x\in E$
$$
P_x(\lim_{n\to\infty }\sigma _{E\setminus E_n}\geq \zeta )=1,
$$
where $\sigma _{E\setminus E_n}$ is the first hitting time of ${E\setminus E_n}$.  
We denote by $S$ the set of smooth, positive Borel measures. 
 In \cite{FOT}, a measure in $S$ is called a {\it smooth measure in the strict sense}. In the sequel,
  we omit the adjective phrase 
``in the strict sense" .

\begin{defi}\label{def-Kato}
Suppose that $\mu\in S$ is a positive smooth measure.
\begin{enumerate}[(1)]
	\item $\mu$ is said to be in the {\it Kato class} of $X$
	($\cK(X)$ in abbreviation)
	if 
	\begin{equation*}
	\lim_{\alpha \to \infty} \| R_{\alpha} \mu \|_{\infty} = 0.
	\end{equation*}
	$\mu$ is said to be in the {\it local Kato class} ($\cK_{loc}(X)$ in abbreviation)
	if for any compact set $K$,  $1_K\cdot \mu$ belongs to $\cK(X)$.	
	\item Suppose that $X$ is transient. A measure $\mu$ is said
	to be in the class $\cK_{\infty}(X)$ if for any $\eps > 0$, 
	there exists a compact set $K = K(\eps)$ 
	\begin{equation*}
	\sup_{x \in E} \int_{K^c} R(x,y)\mu(dy) < \eps.
	\end{equation*}
	$\mu$ in $\cK_{\infty}(X)$ is called {\it Green-tight}. 
\end{enumerate}
\end{defi}

A stochastic process $\{A_{t}\}_{t\geq 0}$ is said to be an {\em additive
functional\/} 
(AF in abbreviation)
if the following conditions hold:

\smallskip
(i) \ $A_{t}(\cdot)$ is ${\cF}_{t}$-measurable for all $t\geq 0$.

(ii) there exists a set $\Lambda\in
       {\cF}_{\infty}=\sigma\left(\cup_{t\geq 0}{\cF}_{t}\right)$ such
       that 
       $P_{x}(\Lambda)=1$, for all $x\in  X$,
       $\theta_{t}\Lambda\subset\Lambda$ for all $t>0$,
       and for each $\omega\in \Lambda$, $A_{\cdot}(\omega)$ is a
       function satisfying:
       $A_{0}=0$, $A_{t}(\omega)<\infty$ for $t<\zeta(\omega)$,
       $A_{t}(\omega)=A_{\zeta}(\omega)$ for $t\geq \zeta $,
       and $A_{t+s}(\omega)=A_{t}(\omega)+A_{s}(\theta_{t}\omega)$ for
       $s,t\geq 0$.

\smallskip
\noindent
If an AF $\{A_{t}\}_{t\geq 0}$ is positive and continuous with respect
to $t$ for each $\omega\in \Lambda$, the AF is called a {\em positive
continuous additive functional\/} (PCAF in abbreviation). The set of all PCAF's is denoted by ${\bf A}^+_c$. 
The family $S$ and ${\bf A}^+_c$ are in one-to-one correspondence ({\it Revuz correspondence}) 
as follows:
for each smooth measure $\mu$, there exists a unique PCAF
$\{A_{t}\}_{t\geq 0}$
such that 
for any $f\in {\mathscr B}^+(E)$ and $\gamma$-excessive
function $h$ ($\gamma\geq 0$),
that is, $e^{-\gamma{t}}p_{t}h\leq h$, 
\begin{equation}\label{eq:7}
 \lim_{t\rightarrow 0}\frac{1}{t}E_{h\cdot{m}}
  \left(\int_{0}^{t}f(X_{s})dA_{s}\right)
  =\int_{E}f(x)h(x)\mu(dx)
\end{equation}
(\cite[Theorem 5.1.7]{FOT}). 
Here, $E_{h\cdot{m}}(\, \cdot\, )=\int_{E}E_{x}(\, \cdot\, )h(x)m(dx)$.
 We denote by $A^{\mu}_t$ the PCAF corresponding to $\mu\in S$.

 \begin{thm}\la{Id}
 Let $\mu\in \cK_{\loc}(X)$. If $(\cE^\mu,\cD(\cE)\cap C_0(E))$ is positive semi-definite, then it is closable.   
Moreover, the semigroup $T^\mu_t$ generated by the closure $(\cE^\mu,\cD(\cE^\mu))$ is expressed as 
$$
T^\mu_tf(x)=p^\mu_tf(x)=E_x\left(e^{A^\mu_t}f(X_t)\right)\ \ m\text{-a.e.}
$$
\end{thm}

\begin{proof}
Let $\{G_n\}$ be a sequence of relatively compact open sets such that 
$G_n\subset \overline{G}_n\subset G_{n+1}$ and $G_n\uparrow E$. Denote by $\mu_n$ the restriction of $\mu$ to $\overline{G}_n$, $\mu_n(\cdot)=\mu(\overline{G}_n\cap\cdot)$. 
 Then since $\mu_n$ is in the Kato class $\cK(X)$, 
 $$
 \cE^{\mu_n}(u)=\cE(u)-\int_E{u}^2d\mu_n,\ \ u\in\cD(\cE)
 $$
is a closed form on $L^2(E;m)$ and the associated $L^2(E;m)$-semigroup $\{T^{\mu_n}_t\}$ equals $\{p^{\mu_n}_t\}$
(\cite[Proposition 3.1]{AM}). 
The sequence of closed, positive form $\{(\cE^{\mu_n},\cD(\cE))\}$ is decreasing in the sense of \cite[p.373]{RS}
and 
$$
\cE^\mu(u)=\lim_{n\to\infty}\cE^{\mu_n}(u),\ \ u\in\cD(\cE).
$$
Hence $(\cE^{\mu_n},\cD(\cE))$ converges to $(\cE^{\ast},\cD(\cE^\ast))$  in strong resolvent sense (\cite[Theorem S.16]{RS}), where $(\cE^{\ast},\cD(\cE^\ast))$ is the closure of the largest closable form smaller
 than $(\cE^\mu,\cD(\cE))$. 
 
 Since $(\cE^\mu,\cD(\cE)\cap C_0(E))$ is positive semi-definite, in particular, lower semi-bounded,
  the semigroup $p^\mu_t$ is strongly continuous on $L^2(E;m)$ (\cite[Theorem 4.1]{AM}).\footnote{In \cite[Theorem 4.1]{AM}, they proved the equivalence between the lower semi-boundedness of $(\cE^\mu,\cD(\cE)\cap L^2(\mu))$ 
 and the strong continuity of $p^\mu_t$. By the regularity of $(\cE,\cD(\cE))$, the lower semi-boundedness of $(\cE^\mu,\cD(\cE)\cap L^2(\mu))$ follows 
 from that of $(\cE^\mu,\cD(\cE)\cap C_0(E))$.}
 For a non-negative Borel function $f$
\begin{align*}
\lim_{n\to\infty}p^{\mu_n}_tf(x)&=\lim_{n\to\infty}E_x\left(e^{A^{\mu_n}_t}f(X_t)\right)
=E_x\left(e^{A^{\mu}_t}f(X_t)\right)=p^{\mu}_tf(x),
\end{align*} 
and thus $(\cE^{\ast},\cD(\cE^\ast))$ is identified with the closed form $(\widetilde{\cE}^\mu,\cD(\widetilde{\cE}^\mu))$ generated by $\{p^\mu_t\}$.

 Define
\begin{align*}
\cD({\cE}_{\overline{G}_n})&=\{u\in \cD({\cE})\mid u=0\ \text{$m$-a.e. on $X\setminus \overline{G}_n$}\}\\
\cD(\widetilde{\cE}_{ \overline{G}_n}^\mu)&=\{u\in \cD(\widetilde{\cE}^\mu)\mid u=0\ \text{$m$-a.e. on $X\setminus  
\overline{G}_n$}\}.
\end{align*}
We then see from \cite[Theorem 5.5]{AM} that  
\begin{align*}
& \text{(i)}\ \cD(\widetilde{\cE}_{ \overline{G}_n}^\mu)=\cD(\widetilde{\cE}_{ \overline{G}_n}^{\mu_n})
=\cD({\cE}_{ \overline{G}_n}),\ \ \text{(ii)}\ \widetilde{\cE}^\mu(u)={\cE}^\mu(u),\ u\in\cD({\cE}_{\overline{G}_n}).
\end{align*}
and that the closure of $\ds{\cup_n\cD(\widetilde{\cE}_{ \overline{G}_n}^\mu)}$
 with respect to $\widetilde{\cE}^\mu_1=\widetilde{\cE}^\mu+(\ ,\ )_m$
is equal to $\cD(\widetilde{\cE}^\mu)$. Therefore, noting that $\cD(\cE)\cap C_0(E)\subset \cup_n\cD(\widetilde{\cE}_{ \overline{G}_n})$, we can conclude that $(\cE^\mu,\cD(\cE)\cap C_0(E))$ is closable and its  closure $\cD(\cE^\mu)$ is identified with $\cD(\widetilde{\cE}^\mu)$. 
\end{proof}

The resolvent $G^\mu_\alpha$  is also probabilistically expressed as
$$
G^\mu_\alpha f(x)=R^\mu_\alpha f(x)=E_x\left(\int_0^\infty e^{-\alpha t}e^{A^\mu_t}f(X_t)dt\right)\ \ m\text{-a.e.}
$$

Let $D$ be an open set. If $\mu$ belongs to  $\cK_{\loc}(X^D)$, the local Kato class 
associated with the part process of $X$ on $D$, Theorem \ref{Id} says that 
the form $(\cE^\mu, \cD(\cE)\cap C_0(D))$ 
 is closable on $L^2(D;m)$. We denote $(\cE^\mu, \cD(\cE^{D,\mu}))$ its closure.  
Theorem \ref{Id} can be extended as follows: 

\begin{thm} \la{cl}
 Let $K\subset E$ be a compact set with {\rm Cap}$(K)=0$. Put $D=E\setminus K$.
 If $\mu$ belongs to the local Kato class $\cK_{\loc}(X^D)$, then 
 the form $(\cE^\mu, \cD(\cE)\cap C_0(E))$ is closable on $L^2(E;m)$ and 
 its closure equals $\cD(\cE^{D,\mu})$.
 \end{thm}
 
 \begin{proof}
First note that since $m(K)=0$,  the form $(\cE^\mu, \cD(\cE)\cap C_0(D))$ 
 is closable on $L^2(E;m)$ and its closure equals $\cD(\cE^{D,\mu})$.  
 
Take a relatively compact open set $G_1\supset K$ and let $(\cE^{G_1},\cD(\cE^{G_1}))$ be the part Dirichlet form 
of $(\cE,\cD(\cE))$ on 
$G_1$. Then $\text{Cap$^{G_1}(K)=0$}$ by \cite[Theorem 4.4.3 (ii)]{FOT}, 
where $\text{Cap$^{G_1}$}$ is the capacity defined 
by $(\cE^{G_1},\cD(\cE^{G_1}))$. Hence, there exists an open set $G_1'$ such that $K\subset G_1'\subset G_1$ and 
$\text{Cap$^{G_1}(G_1')<1$}$. 

Next there exists a relatively compact set $G_2$ such that
 $K\subset G_2\subset \overline{G}_2\subset G_1'$ and the distance between $K$ and 
 $G_2^c$ is less than 1/2,
 $$
 d(K, G_2^c)=\inf\{d(x,y)\mid x\in K,\ y\in G_2^c\}<1/2.
 $$ 
 Since $\text{Cap$^{G_2}(K)=0$}$, there exists an open set $G_2'$ such that $K\subset G_2'\subset G_2$ and 
$\text{Cap$^{G_2}(G_2')<1/2$}$. By repeating this procedure, we have the following sequences of 
open sets $\{G_n\},\ \{G'_n\}$ such that  

\smallskip
i) 
$
G_1\supset G_1'\supset \overline{G}_2\supset G_2\supset G_2'\supset \cdots \supset 
G_n\supset G_n'\supset \overline{G}_{n+1}\supset \cdots\supset K,
$

ii) $\text{Cap$^{G_n}(G_n')<1/n$}$, 

iii) $d(K,G_n^c)<1/n$.

\smallskip
\noindent
Therefore, there exists a sequence $\{\varphi_n\}$ such that $\varphi_n\in \cD(\cE)\cap C_0(G_n)$, 
$0\le \varphi_n\le 1$, $\varphi_n(x)=1$ for $x\in \overline{G}_{n+1}$ and $\cE_1(\varphi_n)<2/n$.
For any $\psi\in\cD(\cE)\cap C_0(E)$ define $\psi_n=\psi-\psi\varphi_n$. Then $\psi_n\in \cD(\cE)\cap C_0(D)$
and $\lim_{n\to\infty}\varphi_n(x)=0$ for any $x\in D$, in particular, $\lim_{n\to\infty}\psi_n(x)=\psi$ $m$-a.e.
Furthermore, since 
$$
\sup_n\cE(\psi_n)^{1/2}\le\sup_n\left(\cE(\psi)^{1/2}+\|\psi\|_\infty\cE(\varphi_n)^{1/2}
+\|\varphi_n\|_\infty\cE(\psi)^{1/2}\right)<\infty,
$$
there exists a subsequence of $\{\psi_n\}$ whose Cesaro mean converges to $\psi$ with respect to $\cE^\mu$.
 Hence the closure $\cD(\cE^{D,\mu})$ of $\cD(\cE)\cap C_0(D)$ contains  $\cD(\cE)\cap C_0(E)$, and thus 
 the form $(\cE^\mu, \cD(\cE)\cap C_0(E))$ is closable and the closure $\cD(\cE^\mu)$ equals $\cD(\cE^{D,\mu})$. 
 \end{proof}
 
  \section{Criticality and Hardy-type inequality}
 For a positive measure $\mu\in\cK_{\loc}$ and a Borel set $B$ of $E$, we 
 denote by $\mu_B$ the restriction of $\mu$ to 
 $B$, $\mu_B(\cdot)=\mu(B\cap\cdot)$.
 For $\mu\in\cK_{\loc}$ define $\nu$ and $\nu_B$ by
 $$
 \nu=\frac{\mu}{R\mu},\quad \nu_B=\frac{\mu_B}{R\mu_B}.
 $$
 For a compact set $K$, the measure $\mu_K$ is in $\cK_{\infty}$ and so
 $R\mu_K$ is bounded continuous (\cite[Proposition 2.2]{C}). Moreover,  
 $\mu_K$ is of finite (0-order) energy integral because 
 \begin{align}\la{domi}
\int_E |\psi|d\mu_K&\le \left(\int_E\psi^2d\mu_K\right)^{1/2}\mu(K)^{1/2} \\
&\le \|R\mu_K\|_\infty^{1/2}\mu(K)^{1/2}\cE(\psi)^{1/2},\ \ \psi\in\cD(\cE)\cap C_0(E)\nonumber
\end{align}
and so $R\mu_K$ belongs to $\cD_e(\cE)$. 
Since   
 \begin{align*}
 \cE^{\nu_K}(R\mu_K,\varphi)&=\cE(R\mu_K,\varphi)-\int_ER\mu_K\cdot\varphi d\nu_K\\
&=\int_E\varphi d\mu_K-\int_E\varphi d\mu_K=0, \ \ \varphi\in\cD(\cE)\cap C_0(E),
 \end{align*}
$R\mu_K$ is a generalized eigenfunction corresponding to the generalized eigenvalue $0$.
 Noting that $\varphi/R\mu_K$ belongs to $\cD(\cE)\cap C_0(E)$ by the same 
 argument as in cite[Lemma 2.4]{T4}, we see from \cite[Theorem 10.2]{FL} that 
 the \Sch\ form $(\cE^{\nu_K},\cD(\cE^{\nu_K}))$ is positive semi-definite
  because
 $$
 \cE^{\nu_K}(\varphi)= \cE^{\nu_K}(R\mu_K(\varphi/R\mu_K))
 =\iint_{E\times E}(R\mu_K)^2d\mu_{\<\varphi/R\mu_K\>}\ge 0.
 $$
  Consequently, $(\cE^\nu,\cD(\cE)\cap C_0(E))$ is 
 also positive semi-definite. In fact, for any $\varphi\in\cD(\cE)\cap C_0(E)$, 
  take a compact set $K$ such that supp$[\varphi]\subset K$
 \begin{align*}
 \cE^{\nu}(\varphi)&=\cE(\varphi)-\int_E\varphi^2\frac{d\mu}{R\mu}
 =\cE(\varphi)-\int_E\varphi^2\frac{d\mu_K}{R\mu}\\
 &\ge\cE(\varphi)-\int_E\varphi^2\frac{d\mu_K}{R\mu_K}=\cE^{\nu_K}(\varphi)\ge 0.
 \end{align*}
 
 Since $R\mu=\lim_{n\to\infty}R\mu_{K_n}$ 
 for a sequence $\{K_n\}$ of compact sets increasing to $E$, $R\mu$ is lower semi-continuous. Since for a 
 non-trivial smooth measure $\mu$,  
  $R\mu(x)>0$ by the irreducibility and so for a compact set $K$, $\inf_{x\in K}R\mu(x)>0$. 
 Hence if $R\mu$ is locally bounded, then  
  $\nu=\mu/R\mu$ is also in $\cK_{\loc}$. 
  Hence we see that for a non-trivial $\mu\in\cK_{\loc}$, 
  $(\cE^\nu,\cD(\cE)\cap C_0(E))$ is closable. We 
   denote $(\cE^\nu,\cD(\cE^\nu))$ its closure.

 \begin{lem} \la{ke} For a compact $K$ and $\mu\in\cK_{\loc}$
$$
\cE^\nu(R\mu_K)\le\iint_{K\times K^c}R(x,y)d\mu(x)d\mu(y).
 $$
 \end{lem}
 
 \begin{proof}
 \begin{align*}
 \cE^\nu(R\mu_K)&=\int_E R\mu_Kd\mu_K-\int_E \frac{(R\mu_K)^2 }{R\mu}d\mu\\
 &=\int_E\left(\frac{R\mu_K(R\mu_K+R\mu_{K^c})-(R\mu_K)^2}{R\mu_K+R\mu_{K^c}}\right)d\mu-\int_E R\mu_Kd\mu_{K^c}\\
 &=\int_E\frac{R\mu_K R\mu_{K^c}}{R\mu_K+R\mu_{K^c}}d\mu-\int_E R\mu_Kd\mu_{K^c}\\
 &=\int_E\frac{R\mu_K R\mu_{K^c}}{R\mu_K+R\mu_{K^c}}d\mu_K+\int_E\frac{R\mu_K R\mu_{K^c}}{R\mu_K+R\mu_{K^c}}d\mu_{K^c}-\int_E R\mu_Kd\mu_{K^c}.
  \end{align*}
Since 
  $$
 \frac{R\mu_K R\mu_{K^c}}{R\mu_K+R\mu_{K^c}}\le R\mu_K,\ \  \frac{R\mu_K R\mu_{K^c}}{R\mu_K+R\mu_{K^c}}\le R\mu_{K^c},
 $$
the right hand side is less than or equal to 
$$
\int_E R\mu_{K^c}d\mu_K+\int_E R\mu_K d\mu_{K^c}-\int_E R\mu_Kd\mu_{K^c}.
$$
Noting 
$$
\int_E R\mu_{K^c}d\mu_K=\int_E R\mu_K d\mu_{K^c}=\iint_{K\times K^c}R(x,y)d\mu(x)d\mu(y),
$$
we have the lemma.
 \end{proof}

 We define a subclass $\cK_{H}$ of $\cK_{\loc}$ as follows: a measure $\mu\in\cK_{\loc}$
 belongs to $\mu\in\cK_{H}$ if $\mu$ satisfies that  $R\mu$ is in $\cD_{\loc}(\cE)\cap \sB_{b,\loc}(E)$
  and there exists an
 increasing sequence $\{K_n\}$ of compact sets such that $K_n\uparrow E$ and 
 \bequ\la{h}
 \sup_n\iint_{K_n\times K_n^c}R(x,y)d\mu(x)d\mu(y)<\infty.
\eequ
 
 \begin{lem}\la{nex}
 If $\mu\in\cK_H$, then $R\mu$ belongs to $\cE_e(\cE^\nu)$. 
 \end{lem}
 
 \begin{proof}
 Let $\{K_n\}$ be a sequence of compact sets for $\mu$ in (\ref{h}).
 Then by Lemma \ref{ke}
 \begin{align*}
 \sup_n\cE^\nu(R\mu_{K_n})<\infty.
 \end{align*}
 Since $\lim_{n\to\infty}R\mu_{K_n}=R\mu$, $R\mu\in\cD_e(\cE^\nu)$. 
 \end{proof}

 \begin{lem}\la{mex} Let $\mu\in\cK_H$. For $\varphi\in\cD(\cE)\cap C_0(E)$
 $$
 \cE^\nu(R\mu,\varphi)=0.
 $$
 \end{lem}
 
 \begin{proof}
  Since $\sup_n\cE^\nu(R\mu_{K_n})<\infty$, There exists a subsequence $\{K_{n_l}\}\subset \{K_n\}$ such that
 $$
 R\left(\frac{({1_{K_{n_1}}+1_{K_{n_2}}\cdots+1_{K_{n_l}}})}{l}\mu\right)\longrightarrow R\mu
 $$
 with $\cE^\nu$-strongly. Let $0\le \phi_l:=(1_{K_{n_1}}+1_{K_{n_2}}\cdots+1_{K_{n_l}})/l\le 1$. Then $\phi\to1$.
 
 For a fixed $\varphi\in\cD(\cE)\cap C_0(E)$ we can assume supp$[\varphi]\subset K_{n_1}$. Then
 \begin{align*}
 \cE^\nu(R\mu,\varphi)&=\lim_{l\to\infty} \cE^\nu(R(\phi_l\mu),\varphi)=\lim_{l\to\infty} \left(\cE(R(\phi_l\mu),\varphi)-
 \int_E \frac{R(\phi_l\mu) }{R\mu}\varphi d\mu\right).
   \end{align*}
Note that $R(\phi_l\mu)$ belongs to $\cD_e(\cE)$. Then since
$$
\lim_{l\to\infty} \cE(R(\phi_l\mu),\varphi)=\lim_{l\to\infty} \int_E \phi_l\varphi d \mu=\int_E\varphi d\mu
$$
and by the monotone convergence theorem 
$$
\lim_{l\to\infty} \int_E \frac{R(\phi_l\mu) }{R\mu}\varphi d\mu=\int_E\varphi d\mu,
$$
we have the lemma.
 \end{proof}

 \begin{thm}\la{h-r}
 If $\mu\in\cK_H$, then $R\mu$ is a ground state of $(\cE^\nu,\cD(\cE^\nu))$, consequently,
 $(\cE^\nu,\cD(\cE^\nu))$ is critical. 
 \end{thm}
 
 \begin{proof}
Since $R\mu$ belongs to $\cE_e(\cE^\nu)$, there exists a sequence $\{\varphi_n\}\subset \cD(\cE)\cap C_0(E)$
such that $\varphi_n$ converges to $R\mu$ $E^\nu$-strongly. Hence
$$
\cE^\nu(R\mu)=\lim_{n\to\infty}\cE^\nu(R\mu,\varphi_n)=0.
$$
 \end{proof}
 
\begin{ex}
Let $X$ be a one-dimensional diffusion process on $(0,\infty)$ with scale function $S$ such that 
$\lim_{x\to 0}S(x)=0,\ \lim_{x\to \infty}S(x)=\infty$. Then 0-resolvent kernel is given by 
 $$
 R(x,y)=S(x)\wedge S(y).
 $$
 Let $\mu\in \cK_{\loc}$ satisfying
 \bequ\la{one}
 \sup_{r>0}\left(\mu((r,\infty))\int_0^rSd\mu \right)<\infty.
 \eequ
 We then have a optimal Hardy inequality
\bequ\la{o-h}
\int_0^\infty u^2\frac{1}{R\mu}d\mu
\le \int_0^\infty\left(\frac{du}{dS}\right)^2dS.
\eequ
If $S(x)=x^p, \ p>0$, then $\mu(dx)=x^{-(p+2)/2}dx$ satisfies (\ref{one}) and the equation (\ref{o-h}) gives us the inequality 
\bequ \la{wh} 
\int_0^\infty\left(x^{-(p+1)/2}\cdot u\right)^2dx\le \frac{4}{p^2}\int_0^\infty\left(x^{-(p-1)/2}\cdot \frac{du}{dx}\right)^2dx,
\eequ
which is the classical one when $p=1$.
  \end{ex}

 \begin{ex}\la{ex5}
 Let us consider the symmetric $\alpha$-stable process $X^{(\alpha)}$
on $\mathbb{R}^d$ with $0 < \alpha < 2$. We assume that
$\alpha < d$, that is, $X^{(\alpha)}$ is transient. 
Let $(\cE^{(\alpha)},\cD(\cE^{(\alpha)}))$ be the Dirichlet form generated by
$X^{(\alpha)}$:
\bequ\la{alpha}
\left\{
\begin{split}
& \cE^{(\alpha)} (u) =
\frac{1}{2} \mathcal{A} (d,\alpha)
\int \!\!\! \int_{\mathbb{R}^d \times \mathbb{R}^d \setminus \triangle }
\frac{(u(x) - u(y))^2}{|x - y|^{d + \alpha}} dx dy \\
& \cD(\cE^{(\alpha)}) =
\left\{ u \in L^2 (\mathbb{R}^d) \mid
\int \!\!\! \int_{\mathbb{R}^d \times \mathbb{R}^d \setminus \triangle }
\frac{(u(x) - u(y))^2}{|x - y|^{d + \alpha}} dx dy < \infty \right\},
\end{split}
\right.
\eequ
where $\triangle =\{(x,x)\mid x\in {\mathbb R}^d\}$ and 
\beqn
\mathcal{A}(d, \alpha) = \frac{\alpha 2^{d-1}\Gamma (\frac{\alpha+d}{2})}
{\pi^{d/2} \Gamma (1 - \frac{\alpha}{2})}
\eeqn
(\cite[Example 1.4.1]{FOT}).

 Let $R(x,y)$ be the $0$-resolvent density of the rotational symmetric $\alpha$-stable process $X$
  on ${\mathbb R}^d$, that is, 
  $$
 R(x,y)=\frac{\Gamma((d-\alpha)/2)}{2^\alpha\pi^{d/2}\Gamma(\alpha/2)}\cdot\frac{1}{|x-y|^{d-\alpha}}.
 $$
We see from Hardy-Littlewood-Sobolev inequality (\cite[Theorem 4.3]{LL}) that for $p, q>1$, $1/p+1/q=(d+\alpha)/d$.
 \begin{align*}
 \iint_{ {\mathbb R}^d\times  {\mathbb R}^d}\frac{1_{B(r)}(x)1_{B(r)^c}(y)}{|x-y|^{d-\alpha}|x|^p|y|^p}dxdy
 &\le C\cdot\|1_{B(r)}|x|^{-\beta}\|_p\|1_{B(r)^c}|x|^{-\beta}\|_q\\
 &=C\cdot r^{d/p-\beta}r^{d/q-\beta}\\
 &=C\cdot r^{d(1/p+1/q)-2\beta},
 \end{align*}
where $B(r)=\{x\in {\mathbb R}^d\mid |x|\le r\}$. 
 Hence if 
 $$
 d\left(\frac{1}{p}+\frac{1}{q}\right)-2\beta=0\ \Longleftrightarrow\ \beta=\frac{d+\alpha}{2},
 $$
 then 
 $$
M:= \sup_{r>0}\iint_{ {\mathbb R}^d\times  {\mathbb R}^d}\frac{1_{B(r)}(x)1_{B(r)^c}(y)}{|x-y|^{d-\alpha}|x|^p|y|^p}dxdy
 <\infty. 
 $$
 
 Let $X^D$ be the part process of $X$ on $D={\mathbb R}^d\setminus\{0\}$. Then the measure $\mu(dx)=|x|^{-(d+\alpha)/2}dx$ belongs to $\cK_H(X^D)$. Indeed, 
 $\mu$ is in $\cK_{\loc}(X^D)$. for $T_n=\{x\in{\mathbb R}^d\mid 1/n\le
 |x|\le n\}$
 \begin{align*}
&\sup_n\iint_{ {\mathbb R}^d\times  {\mathbb R}^d}\frac{1_{T_n}(x)1_{T_n^c}(y)}{|x-y|^{d-\alpha}|x|^{(d+\alpha)/2}|y|^{(d+\alpha)/2}}dxdy\\
\le&\, \sup_n\iint_{ {\mathbb R}^d\times  {\mathbb R}^d}\frac{1_{B(n)}(x)1_{B(n)^c}(y)+1_{B(1/n)}(x)1_{B(1/n)^c}(y)}{|x-y|^{d-\alpha}|x|^{(d+\alpha)/2}|y|^{(d+\alpha)/2}}dxdy\\
\le &\,2M<\infty.
 \end{align*}

 The capacity of $\{0\}$ is zero with respect to the $\alpha$-stable process $X$ and thus 
 the closure of $\cD(\cE^{(\alpha)})\cap C_0(D)$ and that of $\cD(\cE^{(\alpha)})\cap C_0({\mathbb R}^d)$
 is equal. Note that   
 $$ 
 R\mu(x)=\frac{\Gamma((d-\alpha)/4)^2}{2^\alpha\Gamma((d+\alpha)/4)^2}\cdot |x|^{-(d-\alpha)/2}. 
 $$
and so
 $$
 \frac{\mu(dx)}{R\mu(x)}=\kappa^*\frac{1}{|x|^{\alpha}},\ \ 
 \kappa^*=\frac{2^\alpha\Gamma((d+\alpha)/4)^2}{\Gamma((d-\alpha)/4)^2}.
 $$
Then, by Theorem  \ref{h-r}, 
  $R\mu$ 
  is a ground state of 
 $$
\cE^\nu(u)= \cE^{(\alpha)}(u)-\kappa^*\int_{{\mathbb R}^d}u^2/|x|^\alpha dx
 $$
 and  $(\cE^\nu,\cD(\cE^\nu))$ is critical. We see that the constant $\kappa^*$
  equals the best constant of Hardy inequality. 
  \end{ex}
 
\medskip
Let $\mu\in\cK_H$. By Revuz correspondence, $R\mu(x)$ is written as $E_x(A^\mu_\zeta)$ and 
\begin{align*}
R\mu(X_t)&=E_{X_t}(A^\mu_\zeta)=E_x(A^\mu_\zeta(\theta_t)1_{\{t<\zeta\}}|\sF_t)\\
&=E_x((A^\mu_\zeta-A^\mu_t)1_{\{t<\zeta\}}|\sF_t)\\
&=E_x(A^\mu_\zeta|\sF_t)-A^\mu_t, \ \ t<\zeta\ \ P_x\text{-a.s.}
\end{align*}
Define a multiplicative functional $L_t$ by
$$
L_t=\frac{R\mu(X_t)}{R\mu(X_0)}\exp\left(\int_0^t\frac{dA^\mu_s}{R\mu(X_s)}\right).
$$ 
Put
$$
M_t=E_x(A^\mu_\zeta|\sF_t), \ V_t=\exp\left(\int_0^t\frac{dA^\mu_s}{R\mu(X_s)}\right).
$$
Then $R\mu(X_t)=M_t-A^\mu_t$, $t<\zeta$ and by It\^o formula
\begin{align*}
L_t&=1+\frac{1}{R\mu(X_0)}\left(\int_0^tV_sdR\mu(X_s)+\int_0^tR\mu(X_s)V_s\frac{dA^\mu_s}{R\mu(X_s)}\right)\\
&=1+\int_0^tV_sdM_s, \ \ t<\zeta.
\end{align*}
Suppose the Hunt process $X$ has no killing inside. Then the life time $\zeta$ is predictable and 
$\int_0^tV_sdM_s$ becomes a local martingale. Hence
$$
E_x(L_t)=\frac{1}{R\mu(x)}E_x\left(\exp\left(\int_0^t\frac{dA^\mu_s}{R\mu(X_s)}\right)R\mu(X_t)\right)\le 1.
$$
In other words, $R\mu$ is $p^\nu_t$-excessive, $p^\nu_tR\mu\le R\mu$. Hence we have the next corollary.

\begin{cor}
For $\mu\in\cK_H$ let $\nu=\mu/R\mu$ and define 
\begin{equation}
\left\{
\begin{split}
& \cE^{\nu,R\mu} (u,u)  = \cE^{\nu} (R\mu\cdot u, R\mu\cdot u)  \\
& \cD(\cE^{\nu,R\mu}) = \{ u \in L^2(E;(R\mu)^2m)\mid R\mu\cdot u\in \cD(\cE^\nu)\}.
\end{split} 
\right.
\end{equation}
 Then $(\cE^{\nu,R\mu},\cD(\cE^{\nu,R\mu}))$ is a recurrent Dirichlet form.
 \end{cor}
 
\section{Symmetric $\alpha$-stable process: Recurrence and transience}

In this section, we apply the results obtained in previous section to Fractional Schr\"{o}dinger operators  
with Hardy potential.

\begin{ex}\la{exa}
Let us consider the symmetric $\alpha$-stable process $X^{(\alpha)}$
on $\mathbb{R}^d$ with $0 < \alpha < 2$ and $\alpha < d$. 
For $0\le\delta\le d-\alpha$, define a smooth Radon measure $\mu^\delta$ by 
$$
\mu^\delta=\kappa(\delta)|x|^{-\alpha}dx,
$$
where $\kappa(\delta)$ is the constant in (\ref{kapd}). Let 
$$
\delta^*=\frac{d-\alpha}{2},\ \ \ \ \kappa^*=\kappa(\delta^*).
$$
Then $\kappa^*$ is the best constant in the hardy inequality for 
$(\cE^{(\alpha)},\cD(\cE^{(\alpha)}))$ given by (\ref{best}) and
$\mu^{\delta^*}$ equals $\nu$ in Example \ref{ex5}.  
It is known in \cite{BGJP}) that 
$$
\kappa(\delta)<\kappa^*\ \text{for}\  \delta\not=\delta^*,\ \ \ \ 
\kappa(\hat{\delta})=\kappa(\delta)\ \text{for}\ \hat{\delta}=d-\alpha-\delta.
$$
It follows from these facts that 
\bequ\la{sube}
\delta\not=\delta^*\ \Longleftrightarrow\ \lambda(\mu^\delta)>1.
\eequ
By Theorem \ref{ma} and  Lemma \ref{well1}, if $\delta\not=\delta^*$, then 
$(\cE^{\mu^\delta}(:=\cE^{(\alpha), \mu^{\delta}}),\cD(\cE^{\mu^\delta}))$ is subcritical and 
$(\cE^{\mu^{\delta},h},\cD(\cE^{\mu^{\delta},h}))$ is transient for any $h\in\cH^+(\mu^\delta)$. 
We know from \cite[Theorem 3.1]{BGJP} that for $0\le\delta\le\delta^*$
\bequ\la{pre}
p^{\mu^\delta}_t(|x|^{-\delta})=|x|^{-\delta},
\eequ
in particular, the function $|x|^{-\delta}$ belongs to $\cH^+(\mu^\delta)$. 

Let us simply denote $\cE^{\delta}$ for $\cE^{\mu^{\delta},|x|^{-\delta}}$.
It is shown in \cite[Theorem 5.4]{BGJP} that $\cE^{\delta,|x|^{-\delta}}$ is expressed as
\bequ\la{s-d}
\cE^{\delta}(u)=\frac{1}{2} \mathcal{A} (d,\alpha)
\int \!\!\! \int_{\mathbb{R}^d \times \mathbb{R}^d \setminus \triangle }
\frac{(u(x) - u(y))^2}{|x - y|^{d + \alpha}|x|^\delta|y|^\delta} dx dy, \ u\in C^\infty_0(\mathbb{R}^d) 
\eequ
and the closures of $(\cE^{\delta}, C^\infty_0(\mathbb{R}^d))$ 
  in $L^2(\mathbb{R}^d;|x|^{-2\delta}dx)$
is identified with $\cD(\cE^{\delta})$. Moreover, It is shown in \cite[Theorem 5.4]{BGJP} that the closure 
of $(\cE^{\delta}, C^\infty_0(\mathbb{R}^d\setminus\{0\}))$ also is identified with
 $\cD(\cE^{\delta})$,
which implies the capacity of $\{0\}$ with respect to $(\cE^{\delta},\cD(\cE^{\delta}))$ 
equals zero.

As a result, we see that for $0\le\delta<\delta^*$, the Dirichlet form $(\cE^{\delta},\cD(\cE^{\delta}))$
 is transient, which leads us to the transience of $(\cE^{\delta},\cD(\cE^{\delta}))$ 
 for $\delta^*<\delta\le d-\alpha$ (Remark \ref{trans} below).
Note that (\ref{pre}) implies the Markov process generated by 
$(\cE^{\delta},\cD(\cE^{\delta}))$ is conservative.

If $\delta=\delta^*$, then $(\cE^{\mu^{\delta^*}},\cD(\cE^{\mu^{\delta^*}}))$ is critical and 
$(\cE^{\delta^*},\cD(\cE^{\delta^*}))$ is recurrent (\cite{Mi1}). 
Here we give an elementary proof by making an approximation sequence of the identity function $1$. 
First note that $C_0^{\sf lip}(\mathbb{R}^d{\setminus}\{0\}))$ is included in $\cD(\cE^{\delta,|x|^{-\delta}})$,
where $C_0^{\sf lip}(\mathbb{R}^d{\setminus}\{0\}))$ is the set of all Lipschitz 
continuous functions compactly supported in ${\mathbb R}^d{\setminus}\{0\}$.
Let $f_n$ be the function on $[0,\infty)$ such that 
\beqn
f_n(t)=\left\{
\begin{split}
1, \quad & \quad 0\le t \le n, \\
2n-t, & \quad n\le t \le 2n,\\
0,\quad  &\quad t\ge 2n
\end{split}
\right.
\eeqn
and put $\phi_n(x)=f_n(|x|)$. We define the function $\varphi_n(x)$ by
\begin{equation} \label{test0}
\varphi_n(t)=\left\{
\begin{split}
\phi_n(x), \quad \ \ & \quad |x|\ge 1, \\
\phi_n(Tx),\quad  &\quad 0< |x|\le 1,
\end{split}
\right.
\end{equation}
where $T$ is a map from $\mathbb{R}^d{\setminus}\{0\}$ to $\mathbb{R}^d{\setminus}\{0\}$ 
defined by 
\bequ\la{kel}
Tx=x/|x|^2.  
\eequ
Then $\varphi_n \in C_0^{\sf lip}({\mathbb R}^d{\setminus}\{0\})$ for 
each $n$.

Since $\varphi_n(x)\uparrow 1$ and $({\mathcal E}^{\delta, |x|^{-\delta}}, 
{\mathcal D}({\mathcal E}^{\delta, |x|^{-\delta}}))$ is the Dirichlet form of pure-jump type,  
it is enough to see the following estimate in order to 
obtain the recurrence of the form (see \cite{U02}): 
\bequ\la{or}
\sup_n \iint_{\mathbb{R}^d \times \mathbb{R}^d \setminus \triangle }
\frac{(\varphi_n(x) - \varphi_n(y))^2}{|x - y|^{d + \alpha}|x|^{(d-\alpha)/2}|y|^{(d-\alpha)/2}} dx dy<\infty.
\eequ

\noindent
Put
$$
\Phi_n(x,y)=\frac{(\varphi_n(x) - \varphi_n(y))^2}{|x - y|^{d + \alpha}|x|^{(d-\alpha)/2}|y|^{(d-\alpha)/2}} 
\ \  {\rm for} \  (x,y) \in {\mathbb R}^d \times {\mathbb R}^d  \ {\rm with} \ x\not=y.
$$
The map $T$ satisfies 
$$
|Tx - Ty|^{d + \alpha}|Tx|^\delta|Ty|^\delta=|x-y|^{d + \alpha}|x|^{-d-\alpha-\delta}|y|^{-d-\alpha-\delta},
$$
and the Jacobian of $T$ equals $1/|x|^{2d}$. As a result,  we see from the definition of $\varphi_n$ that
$$
\iint_{\{|x|\le1\} \times \{|y|\le1\}}
\Phi_n(x,y) dx dy=\iint_{\{|x|\ge1\} \times \{|y|\ge1 \}}\Phi_n(x,y)dx dy,
$$
and so 
\begin{align*}
&\iint_{\mathbb{R}^d \times \mathbb{R}^d} \Phi_n(x,y) dx dy \\
& =\,2\left(\iint_{\{|x|\le1\} \times \{|y|\ge1 \}}\Phi_n(x,y)dx dy
+ \iint_{\{|x|\ge1\} \times \{|y|\ge1 \}}\Phi_n(x,y)dx dy\right)\\
&=:2({\rm I}+{\rm II}).
\end{align*}
Since $\Phi_n(x,y)=0$ for $(x,y) \in \{1/n \le |x|\le 1\} \times \{1\le |y|\le n\}$,  
the integral  in {\rm I} can be decomposed into two parts as follows:
\begin{align*}
{\rm I} & = \iint_{\{|x|\le1/n\} \times \{|y|\ge1 \}}\Phi_n(x,y)dx dy
+\iint_{\{1/n\le|x|\le1\} \times \{|y|\ge n \}}\Phi_n(x,y)dx dy \\
& =:  {\rm I}_1+{\rm I}_2.
\end{align*}
The inequality $|x-y|\ge \left(1-{|x|}/{|y|}\right)|y|\ge\left(1-{1}/{n}\right)|y|$ holds 
for 
$$
(x,y)\in \Big(\{|x|\le1/n\} \times \{|y|\ge1 \} \Big)\bigcup 
\Big(\{1/n\le|x|\le1\} \times \{|y|\ge n\}\Big).
$$   
So we have $\Phi_n(x,y)\le c|x|^{-(d-\alpha)/2}|y|^{-(3d+\alpha)/2}$ and ${\rm I}_1$ 
and ${\rm I}_2$ converges to 0 as $n\to\infty$.

Noting that $\Phi_n$ is a symmetric function on ${\mathbb R}^d \times{\mathbb R}^d$ and 
vanishes on the set $\{|x|\ge 3n\}\times\{|y|\ge 3n\}$,  
divide the integral in {\rm (II)} into two parts:
\begin{align*}
{\rm II} &= 2 \hspace*{-10pt} 
\iint\limits_{\{1\le |x|\le 3n\} \times \{ 3n \le |y| \}}\Phi_n(x,y)dx dy
+ 2\iint\limits_{\{1\le  |x| \le  |y| \le 3n \}}\Phi_n(x,y)dx dy \\
&  =: 2 ({\rm II}_1+{\rm II}_2).
\end{align*}
Since  $\varphi_n$ vanishes on the set $\{2n \le |x| \}$ 
and the estimates  $|x-y|\ge (1/3)|y|$ hold for $(x,y)\in\{1\le |x|\le 2n\} \times \{|y|\ge 3n \}$,  
we see that 
$$
{\rm II}_1\le c\iint\limits_{\{1\le |x|\le 2n\} \times \{|y|\ge 3n \}}
|x|^{-(d-\alpha)/2}|y|^{-(3d+\alpha)/2}dx dy\le C.
$$
The Lipschitz continuity $|\varphi_n(x)-\varphi_n(y)|\le (1/n^2)|x-y|$ tells us 
\begin{align*}
{\rm II}_2&\le c/n^2\iint\limits_{\{1\le |x|\le |y|\le 3n \}}
|x|^{-(d-\alpha)/2}|y|^{-(d-\alpha)/2}|x-y|^{-(d+\alpha-2)}dxdy \\
&\le c/n^2\int_{\{1\le |x|\le 3n \}}|x|^{-(d-\alpha)}dx\int_{\{|y|\le 6n \}}|y|^{-(d+\alpha-2)}dy\le C.
\end{align*}
Therefore we have the estimate (\ref{or}).
Thus we find that $({\mathcal E}^{\delta^*}, 
{\mathcal D}({\mathcal E}^{\delta^*}))$ is recurrent.

\end{ex}

\begin{rem} \la{trans} Suppose $\delta\not=\delta^*$. By the Sobolev inequality, 
$$
L^{(2d)/(d+\alpha)}\subset \cL.
$$
Hence by Lemma \ref{e-5} $G^\mu f\in \cD_e(\cE^\mu)$ for any $f\in L^{(2d)/(d+\alpha)}$.
\end{rem}

\begin{rem} \la{trans} We know form Example \ref{exa} that 
$$
\sup_n\cE^{\mu^{\delta^\ast}}(|x|^{-\delta^\ast}\varphi_n)=\sup_n\cE^{\delta^\ast}(\varphi_n)<\infty,
$$
and $|x|^{-\delta^\ast}\varphi_n(x) \uparrow |x|^{-\delta^\ast}$ as $n\to\infty$ for $x\not=0$. 
Hence, the function $|x|^{-\delta^\ast}$ belongs to $\cD_e(\cE^{\mu^{\delta^\ast}})$ and $\cE^{\mu^{\delta^*}}(|x|^{-\delta^\ast})=0$, 
that is,  it is the ground state. Therefore, we see that $(\cE^{\mu^{\delta^\ast}},
\cD(\cE^{\mu^{\delta^\ast}}))$ is critical in the sense of Definition \ref{def-criticality} (2). 

Since 
 $$
 \int_{\mathbb{R}^d}|x|^{-2\delta^\ast}|x|^{-\alpha}dx= \int_{\mathbb{R}^d}|x|^{-d}dx=\infty,
 $$
the function $|x|^{-\delta^\ast}$ does not belong to $\cD_e(\cE^{(\alpha)})$ and thus $\cD_e(\cE^{(\alpha)})
\subset\cD_e(\cE^{\mu^{\delta^\ast}})$. We see that the minimizer of $\lambda(\mu^{\delta^\ast})(=1)$ does not 
exist, in particular, $(\cD_e(\cE^{(\alpha)}),\cE^{(\alpha)})$ is not embedded in $L^2(\mathbb{R}^d,|x|^{-\alpha}dx)$.
\end{rem}

\begin{rem} \la{Inv} Denote by $X^\delta$ the Markov process generated by the regular 
Dirichlet form $(\cE^{\delta},\cD(\cE^{\delta}))$ on $L^2(\mathbb{R}^d,1/|x|^{2\delta}dx)$.
Let $T$ be the map in (\ref{kel}). Then the transformed process $T(X^\delta_t)$ is symmetric 
with respect to $1/(|x|^{2d-2\delta})dx$ and the  
 Dirichlet form generated by $T(X^\delta_t)$ is written as 
$$
\frac{1}{2} \mathcal{A} (d,\alpha)
\int \!\!\! \int_{\mathbb{R}^d \times \mathbb{R}^d \setminus \triangle }
\frac{(u(x) - u(y))(v(x) - v(y))}{|x - y|^{d-\alpha}|x|^{d-\alpha-\delta}|y|^{d-\alpha-\delta}} dx dy,
$$
that is, the Dirichlet form $(\cE^{\hat{\delta}},\cD(\cE^{\hat{\delta}}))$ 
on $L^2(\mathbb{R}^d,1/|x|^{2d-2\delta}dx)$. Noting $2d-2\delta=2\hat{\delta}+2\alpha$, 
we see that the time-changed process of $X^{\hat{\delta}}_t$ by $\int_0^t1/|X^{\hat{\delta}}_s|^{2\alpha} ds$ 
is identified with $T(X^\delta_t)$. Let $\tau_t=\inf\{s>0\mid \int_0^s(1/|X^{\hat{\delta}}_u|^{2\alpha} du>t\}$.
Then $X^{\hat{\delta}}_{\tau_t}$ has the same law as that of $T(X^\delta_t)$. 
In other words, $X^\delta_t$ has the inversion property (cf, \cite{ACGZ}).
By the time-change,
the recurrence and transience are invariant, and thus the transience of $X^\delta$ implies that of $X^{\hat{\delta}}$.

\end{rem}

 \end{document}